\newcommand{\Weierstra}{Weierstrass}
\newcommand{\Gau}{Gauss}

\documentclass{article}

\usepackage[dvips]{graphicx}

\begin{document}

\font\bbbld=msbm10

\newcommand{\bfA}{\hbox{\bbbld A}}
\newcommand{\bfR}{\hbox{\bbbld R}}
\newcommand{\bfC}{\hbox{\bbbld C}}
\newcommand{\bfZ}{\hbox{\bbbld Z}}
\newcommand{\bfH}{\hbox{\bbbld H}}
\newcommand{\bfQ}{\hbox{\bbbld Q}}
\newcommand{\bfN}{\hbox{\bbbld N}}
\newcommand{\bfP}{\hbox{\bbbld P}}
\newcommand{\bfT}{\hbox{\bbbld T}}
\newcommand{\suchthat}{\mid}
\newcommand{\halo}[1]{\Int(#1)}
\newcommand{\union}{\cup}
\newcommand{\goesto}{\rightarrow}
\newcommand{\bdy}{\partial}
\newcommand{\n}{\noindent}
\newcommand{\tensor}{\otimes}
\newcommand{\smooth}{\hbox{$C^{\infty}$}}
\newcommand{\ds}{\displaystyle}
\newcommand{\tin}{\hbox{$\,\in\,$}}
\newcommand{\po}[1]{\frac{\pi}{#1}}
\newcommand{\dpo}[1]{\hbox{$\ds\frac{\pi}{#1}$}}

\def\Sym{\mathop{\rm Sym}}

\def\Int{\mathop{\rm Int}}
\def\Re{\mathop{\rm Re}}
\def\Im{\mathop{\rm Im}}

\newtheorem{theorem}{Theorem}[section]
\newtheorem{assertion}{Assertion}[section]
\newtheorem{proposition}{Proposition}[section]
\newtheorem{remark}{Remark}[section]
\newtheorem{lemma}{Lemma}[section]
\newtheorem{definition}{Definition}[section]
\newtheorem{claim}{Claim}[section]
\newtheorem{corollary}{Corollary}[section]
\newtheorem{observation}{Observation}[section]
\newtheorem{conjecture}{Conjecture}[section]
\newtheorem{question}{Question}[section]
\newtheorem{example}{Example}[section]
\newtheorem{normalization}{Normalization}[section]

\newbox\qedbox

\newenvironment{proof}{\smallskip\noindent{\bf Proof.}\hskip \labelsep}%
			{\hfill\penalty10000\copy\qedbox\par\medskip}
\newenvironment{proofspec}[1]%
		      {\smallskip\noindent{\bf Proof of #1.}\hskip \labelsep}%
			{\nobreak\hfill\hfill\nobreak\copy\qedbox\par\medskip}
\newenvironment{acknowledgements}{\smallskip\noindent{\bf Acknowledgements.}%
	\hskip\labelsep}{}

\setlength{\baselineskip}{1.3\baselineskip}

\renewcommand{\baselinestretch}{1.0}

\newcounter{i}

\title{Embedded, Doubly--Periodic Minimal Surfaces}

\author{Wayne Rossman, 
        Edward C. Thayer\thanks{Supported by the National Science Foundation
	under grants DMS-9011083 and DMS-9312087 and by the U.S. Department
	of Energy under grant DE-FG02-86ER25015 of the Applied Mathematical
	Science subprogram of the Office of Energy Research}, 
        Meinhard Wohlgemuth\thanks{Supported by SFB 256 at University of Bonn
	and the Alexander von Humboldt-Stiftung}}
\date{}
\maketitle

\begin{abstract}
We consider the question of existence of embedded doubly periodic minimal 
surfaces in $\bfR^3$ with Scherk--type ends, surfaces that topologically are 
Scherk's doubly periodic surface with handles added in various ways.  
We extend the existence results of H. Karcher and F. Wei to more 
cases, and we find other cases where existence does not hold.  
\end{abstract}

\section{Introduction}

 H. Karcher \cite{kar3} proved
the existence of the 
first complete, embedded, doubly--periodic minimal 
surface to be found since H. Scherk's classical 
example, which dates from 1835.
 We denote Karcher's surface by $M_1$ (see Figure \ref{fig:intro_images}, 
left).
 Following this discovery, Wei \cite{wei2} constructed
an embedded, doubly--periodic surface of genus two by
adding a handle to $M_1$ (Figure \ref{fig:intro_images}, center).
 We describe a new embedded, genus two surface that 
results from adding a different type of handle to $M_1$ (Figure 
\ref{fig:intro_images}, right),
and outline the differences between these two genus two 
surfaces.
 In addition, we construct three collections of new, embedded
surfaces of genus three that result from adding either two handles 
of the same type (Figure \ref{fig:2k+1pp_image}) or two handles of 
different type (Figure \ref{fig:intro_images2}).  

 Using a technique discovered by Karcher-Polthier \cite{kp} 
to reduce the number of periods to be considered, we
are able to add ends to the fundamental piece of
each surface presented without increasing the dimension of the 
period problems, 
thereby producing countably many different families 
of new, embedded examples for each of the handle types shown in 
Figures \ref{fig:intro_images} and \ref{fig:2k+1pp_image}.  

 The existence proofs for the genus two surfaces require
solving one--dimensional period problems, and 
 the existence proofs for the genus three surfaces require
solving either one--dimensional or two--dimensional period problems, 
depending on the types of handles we choose.  When the period 
problem is one--dimensional (as for surfaces in 
Figures \ref{fig:intro_images} and \ref{fig:2k+1pp_image}), 
we use the intermediate value theorem 
to solve it.  When it is two--dimensional (as for surfaces 
in Figure \ref{fig:intro_images2}), we achieve a solution by 
using a mapping degree argument, a kind of 
generalization of the intermediate value theorem.   

We find that in the two cases of genus three surfaces with 
four ends and handles of the same type 
the period problems have no solution.  In these 
exceptional cases we demonstrate a natural geometric obstruction 
to existence, an obstruction that disappears when more ends are 
added to the surfaces.  

\begin{figure}
\begin{center}
  \includegraphics[width=7.5cm]{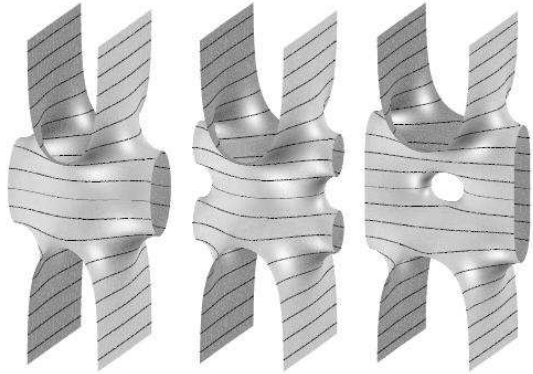} 
\end{center}
   \caption{Fundamental pieces of Karcher's surface $M_1$(left), 
            Wei's surface $M_1^-$(center), and 
 		$M_1^+$ (right).}		\label{fig:intro_images}
\end{figure}

\begin{figure}
\begin{center}
  \includegraphics[width=12.5cm]{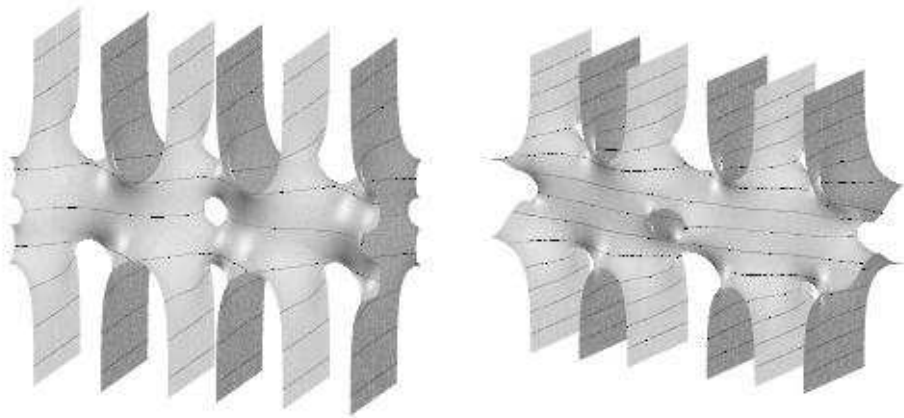} 
\end{center}
   \caption{Fundamental pieces of $M_3^{--}$ and $M_3^{++}$}
                                         	\label{fig:2k+1pp_image}
\end{figure}

\begin{figure}
\begin{center}
  \includegraphics[width=7.5cm]{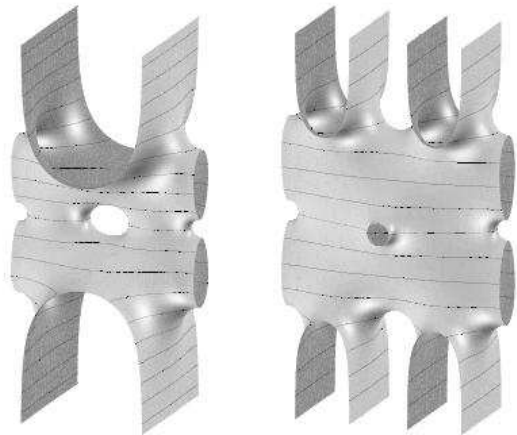} 
\end{center}
   \caption{Fundamental pieces of $M_1^{+-}$ (left), and
		$M_2^{+-}$ (right).}		\label{fig:intro_images2}
\end{figure}

\setcounter{equation}{0}
\section{Overview of the construction}
\begin{figure}
\begin{center}
  \includegraphics[width=12.5cm]{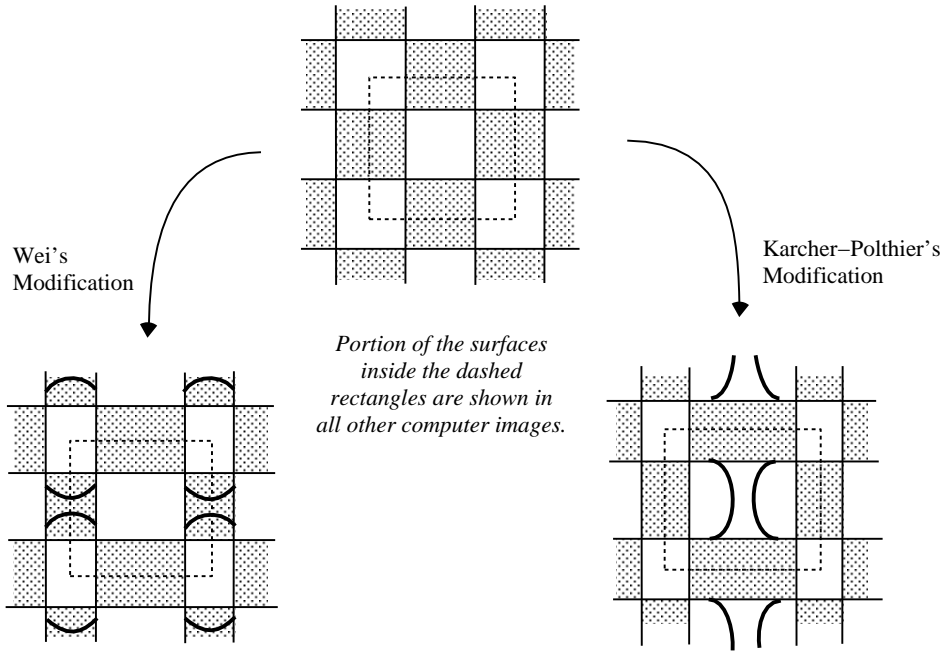} 
\end{center}
   \caption{$M_1$, $M_1^-$, and $M_1^+$ projected onto the
 	$x_1$-$x_2$ plane.}			\label{fig:wt_top_view}
\end{figure}

\begin{figure}
\begin{center}
  \includegraphics[width=4.5cm]{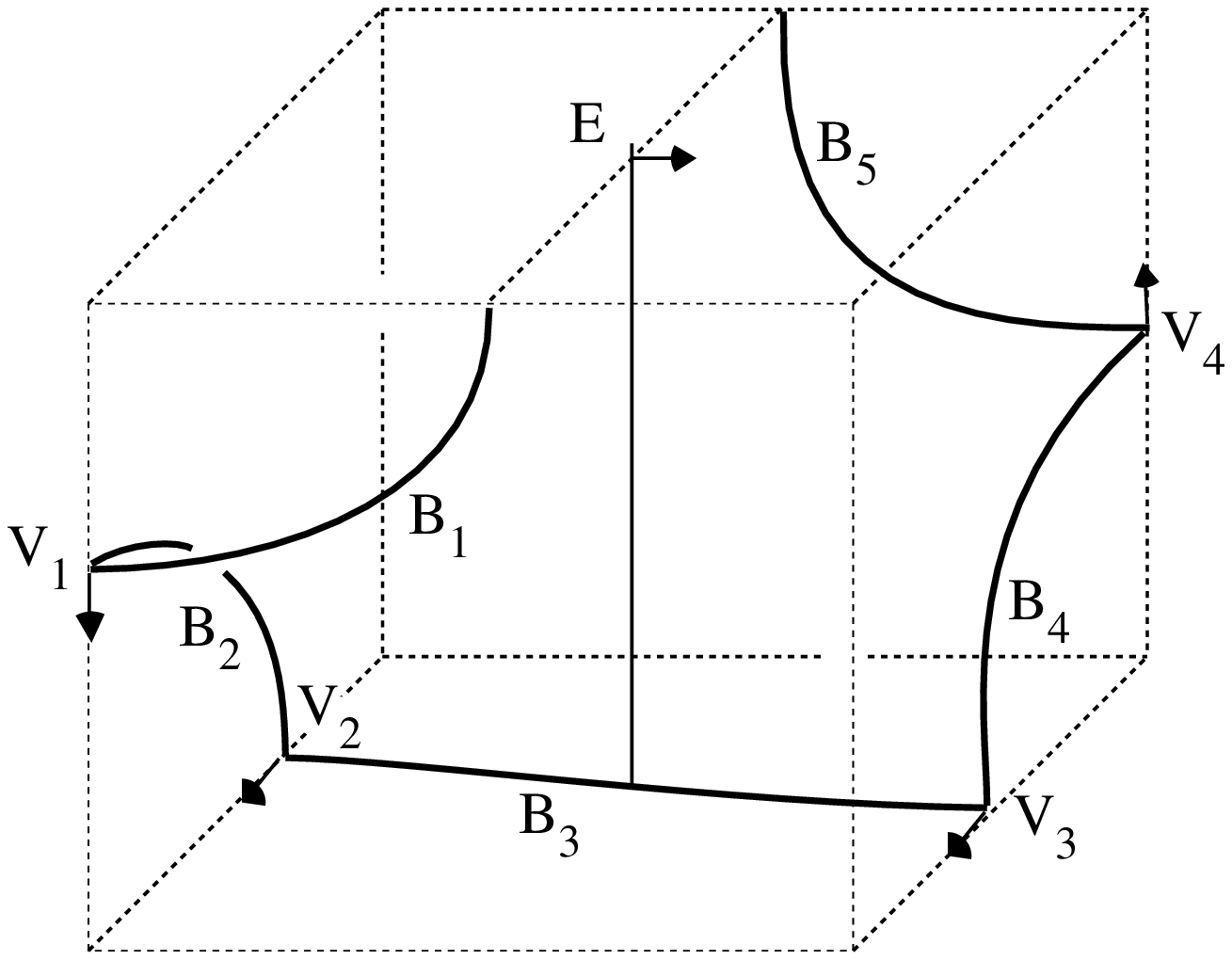} 
  \includegraphics[width=4.5cm]{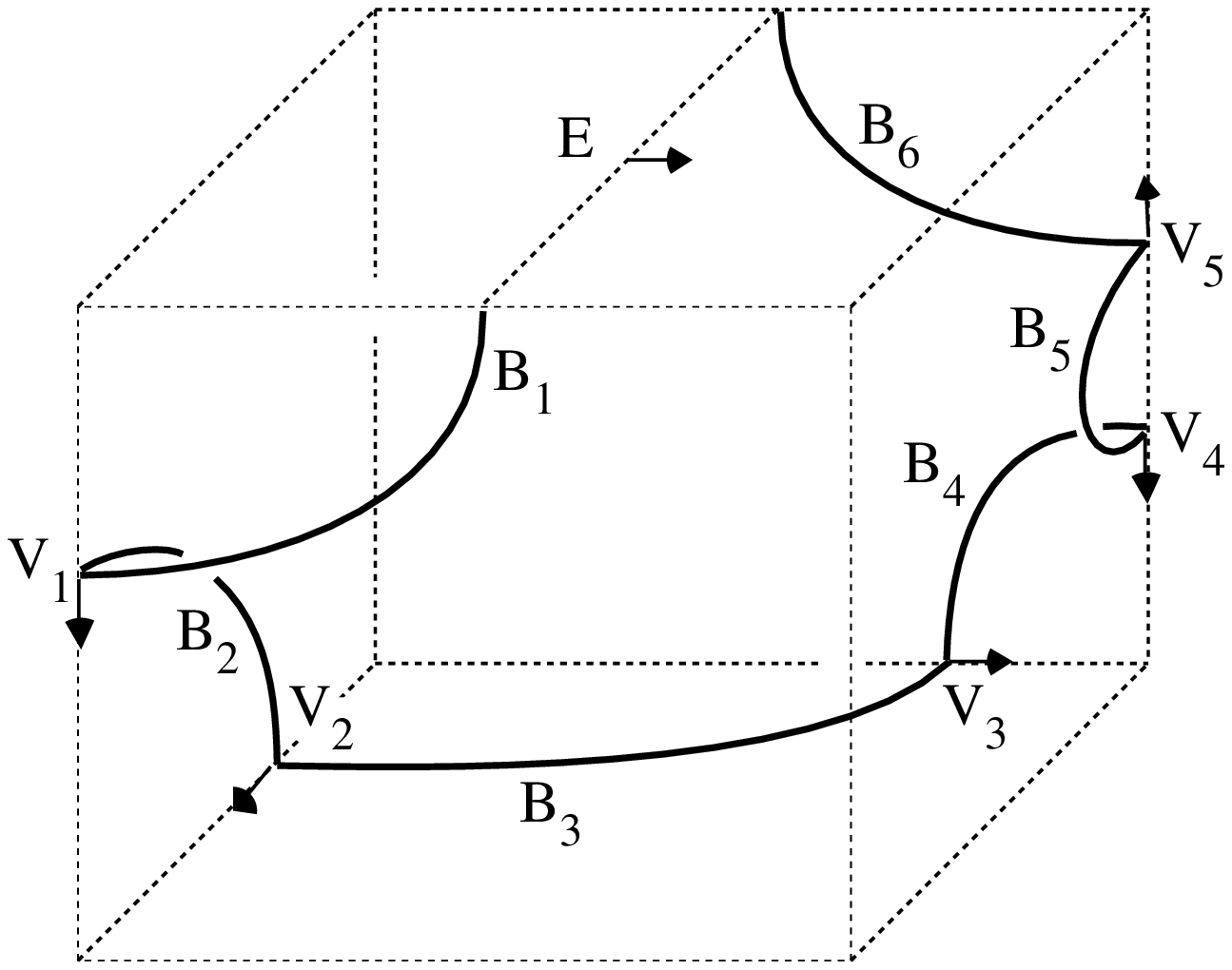} 
  \includegraphics[width=4.5cm]{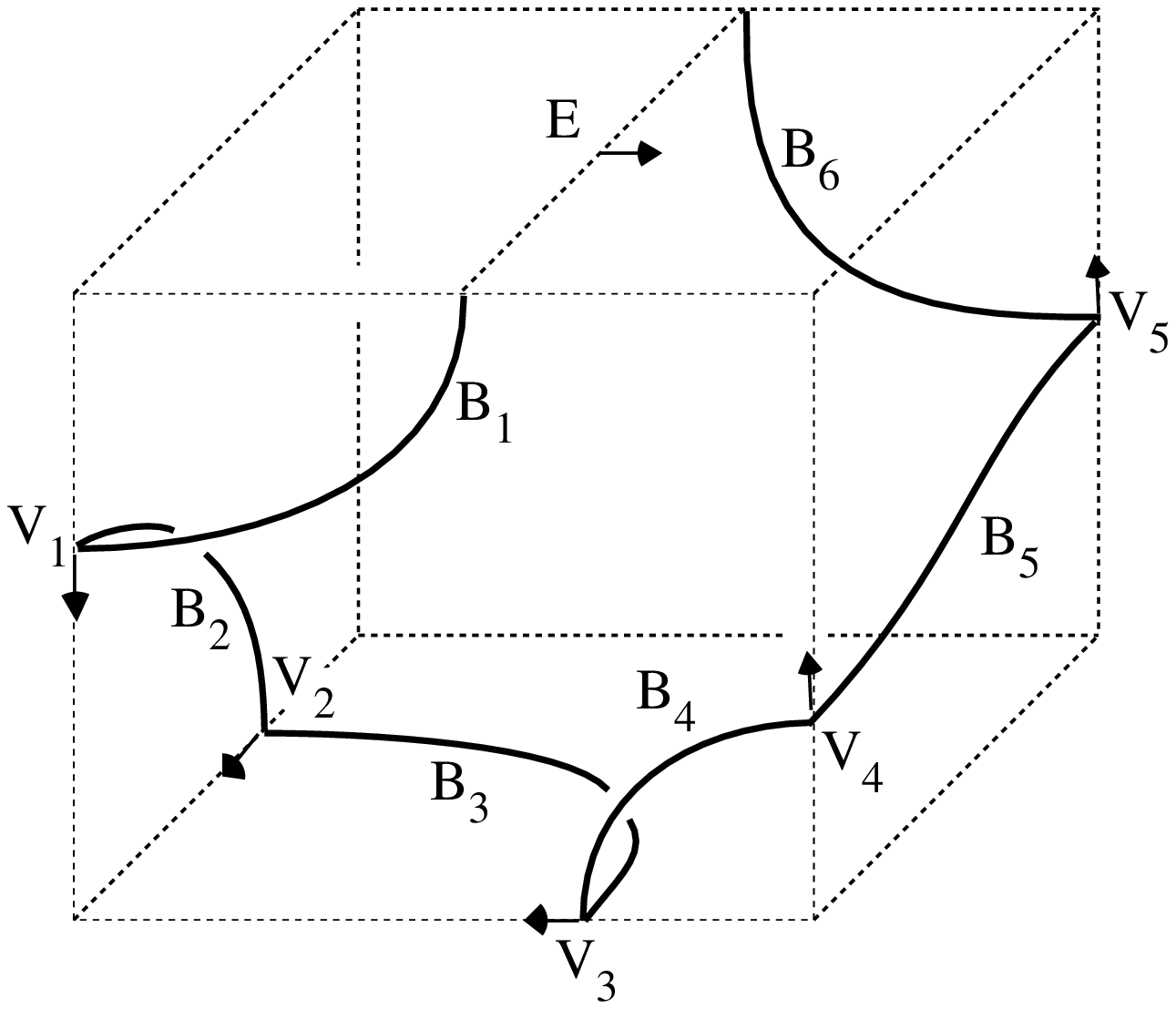} 
\end{center}
   \caption{Sketches of one eighth of $M_1$(left), $M_1^-$(center),
 		and $M_1^+$ (right).}		\label{fig:intro_sketches}
\end{figure}

 Karcher's original surfaces $M_1$ are highly symmetric;
they have three
mutually perpendicular planes of symmetry and 
contain four vertical straight lines (see Figures \ref{fig:intro_images}, 
left and \ref{fig:intro_sketches}, left).  
 The three planes divide the surface into eight pieces.
 Each piece is bounded by planar geodesic curves, and
has one end.
 Since all the surfaces we will discuss here share 
these planar symmetries we will focus on one eighth of 
these surfaces and draw sketches of this portion only.

The first modification of $M_1$ was made by 
F. Wei \cite{wei2}, who constructed a 
one--parameter family of genus two examples 
$M_1^-$ by adding a single handle over
one of the two saddle points of $M_1$ (see Figures 
\ref{fig:intro_images}, center and
\ref{fig:intro_sketches}, center).  
 Recently it was discovered by Karcher--Polthier \cite{kp}
(and the second author independently) that another
modification of $M_1$ was possible.
 This new surface $M_1^+$ results from adding a handle 
to $M_1$ in a different direction, thereby producing
another doubly--periodic, embedded minimal surface
of genus two (see Figures \ref{fig:intro_images}, right and 
\ref{fig:intro_sketches}, right).  

 {\bf Remark on notation:} {\em
In order to distinguish the two genus two surfaces, 
we view $M_1$ from above, 
imagining that $M_1$ projects into the black squares
of an infinite, black and white checkerboard
pattern, with the vertical straight lines projecting onto 
the corners of these squares (see Figure \ref{fig:wt_top_view}).
 From this perspective, the handles added by Wei 
project into the black regions while the new handles 
project into the white ones.
 In both cases, the additional handles modify the 
checker board pattern into a tiling made up of rectangular
regions as is indicated in Figure \ref{fig:wt_top_view}.
 We denote the handles over the black squares with
a superscript `--', and those over the white squares
with a superscript `+'.
 Hence, in this notation, Wei's genus two surface
is referred to as
$M_1^-$, and the new surface discussed in Section \ref{sec:m_k+}
is $M_1^+$.
\newline
 \hspace{0.16in} Each surface discussed in this paper lies in a 
   one--parameter family of embedded surfaces.
    Since we are interested in the topological 
   qualities of these surfaces, 
   our notation thoughout the paper will not reflect the 
   specific surface in the family.
    The subscript indexes the number of ends on each 
   eighth of the surface.
}

 Both $M_1^+$ and $M_1^-$ have smaller symmetry groups 
than Karcher's original surface; in particular, the 
vertical straight lines of $M_1$ are eliminated.
 The question ``Is it possible to add handles to $M_1$
and preserve the original symmetries?'' is a natural one.
 We might, for example, want to add either a `+' or a `--'
type handle and preserve the vertical straight lines.
 Rotation about these lines (via 
the Schwarz reflection principle, Theorem 
\ref{thm:schwarz}) places
another handle over the other saddle point of $M_1$.  
 This would result in a genus three surface with
four Scherk ends.
 It is easy to imagine such a surface for either type 
of handle.
 Indeed, the suggested conjugate contour of one eighth of 
either surface supports a Plateau solution that is a 
Jenkins-Serrin graph \cite{jes1}.  
 So a minimal surface with boundary exists with 
the desired shape, but we only know that 
certain bounding planar curves lie in parallel planes.
 We then must solve the one--dimensional period problem 
or, equivalently, show that the parallel planes coincide.
 We will prove that neither of these period problems are 
solvable, and we do so 
by finding natural obstructions on the corresponding conjugate 
surfaces.
 Understanding these obstructions, we realize it is possible
to overcome them by adding more ends to each surface.
 Because of the desired symmetries, each eighth of these
surfaces must have an odd number of ends.
 Indexing by this number, we show the period problems are
never solved on $M_1^{--}$ and $M_1^{++}$, and 
that for $k \ge 1$ the period problems associated to 
$M_{2k+1}^{--}$ and $M_{2k+1}^{++}$ are solvable.
 The superscript indicates the types of handles added
to $M_1$.  See Figure \ref{fig:2k+1pp_image}.  

 With the addition of each new end, there is a new
associated period.
 In Section \ref{sec:m_k}, we describe a technique found by 
Karcher--Polthier \cite{kp} that shows that one may 
simultaneously solve these end periods.
 Specifically, they observed that a certain simple 
restriction on the 
conjugate contours ensures these end periods are all zero.
 Moreover, this restriction leaves an ample number
of parameters free to allow us to adjust the other periods
associated with the new handles.

 Instead of adding two handles of the same type to $M_1$,
we may also consider surfaces which have 
two handles of different types.
 This produces a family of genus three surfaces 
that no longer have the straight line symmetries of $M_1$.
 Without this additional symmetry, the period problem 
resulting from the new handles is two--dimensional.
 The third author's experience with two--dimensional 
period problems \cite{w} suggested that these period
problems may be solvable.  We 
prove in Section \ref{sec:m_k+-} that $M_{1}^{+-}$ with 
four Scherk--type 
ends exists.  Generalizing the examples $M_{1}^{+-}$ to 
have $4k$ Scherk--type ends 
for $k \ge 2$, numerical evidence suggests the 
existence of $M_{k}^{+-}$ for $k \ge 2$ (see Figure 
\ref{fig:intro_images2}).  

The computer graphics in the figures were created using the MESH software 
produced by James T. Hoffman of the Mathematical Sciences Research Institute, 
Berkeley, California, U.S.A..  

\setcounter{equation}{0}
\section{Background results needed for the construction}	\label{sec:wt_back}
\bigskip
We consider only connected and properly immersed minimal surfaces.
 To establish notation we state the following description 
of the \Weierstra\ Representation.

\bigskip

{\bf \Weierstra--Representation:} \space Let $M$ be
a minimal surface in $\bfR^3$ and $R$ the underlying 
Riemann surface of $M$.
 Then $M$ can be expressed, up to translations, 
in terms of a meromorphic
function $g$ on $R$, the so-called \Gau\ map (since $g$ will be 
stereographic projection of the oriented normal vector of $M$),
and a holomorphic differential $\eta$ on $R$ by
\begin{equation}					\label{eq:WSD1}
   F(p) = Re \, \int_{p_0}^p \, (\phi_1,\phi_2,\phi_3) \; \; \; \; , 
\end{equation}
where $p_0 \in R$ is fixed and
\begin{equation}					\label{eq:WSD2}
   (\phi_1,\phi_2,\phi_3) = \left( \Big( {1\over g} - g \Big) \eta,
	\, i \Big( {1\over g} + g \Big) \eta, \, 2\eta \right) \; .
\end{equation}
 Conversely, let $R$ be a Riemann surface, $g$ a
meromorphic function on $R$, and $\eta$ a holomorphic
differential on $R$. 
 Then (\ref{eq:WSD1}) and (\ref{eq:WSD2}) 
define a conformal
minimal immersion $F: \; R \to \bfR^3$, provided 
the poles and zeros of order $\ell$ of $g$ coincide with the zeros of 
order $\ell$ of $\eta$, and $(\phi_1,\phi_2,\phi_3)$ has no real periods, i.e.
\begin{equation}					\label{eq:WSD3}
   \hbox{Period}_{(\phi_1,\phi_2,\phi_3)}(\gamma) = \int_\gamma 
      (\phi_1,\phi_2,\phi_3) \in i\bfR
\end{equation}
for all closed curves $\gamma$ on $R$. 

\bigskip

 We call $(R,g,\eta)$ the {\em \Weierstra\ data} of the 
minimal surface $M$.  Denoting the universal cover of $R$ by 
$\tilde{R}$, the minimal immersion $F^*:\tilde{R} \to 
\bfR^3$ with the \Weierstra\ data 
$(R,g,i\eta)$ is called 
the {\em conjugate surface} to $M$, and is denoted by $M^*$.  
It is known that any curve of $R$ which is mapped by $F$ to 
a non-straight planar geodesic of $M$ is mapped by $F^*$ to 
a straight line in $M^*$.  
Furthermore, since the \Gau\ map $g$ and the first fundamental 
form are the same for both $M$ and $M^*$, it follows that 
the planar geodesic in $M$ will lie in a plane perpendicular 
to the corresponding line in $M^*$ and that the planar 
geodesic in $M$ will have the same length as the line in 
$M^*$.  We will use these properties 
extensively.  The following known results are also central 
to the arguments we will be making.

\begin{theorem}                 \label{thm:schwarz}
{\em (Schwarz reflection principle)}  If a minimal surface 
$M \subset \bfR^3$ with boundary contains a non-straight 
planar geodesic $\cal C$ (resp. straight line $\cal C$)
in its boundary, then $M$ can be extended smoothly 
across $\cal C$ by reflecting through the plane containing 
$\cal C$ (resp. rotating about $\cal C$).  
\end{theorem}

\begin{theorem}{\em (Krust \cite{dhkw1})}			\label{thm:krust}
    If an embedded minimal surface $F:B\goesto\bfR^3$,
   $B=\{w\tin\bfC : |w| < 1\}$ can be written as a
   graph over a convex domain in a plane, then the
   conjugate surface $F^*:B\goesto \bfR^3$ is also a graph
   over a domain in the same plane.
\end{theorem}

\begin{theorem}{\em (Jenkins-Serrin \cite{jes1})}		\label{thm:Jen-Ser}
    Let $D$ be a bounded convex domain such that $\partial D$
   contains two sets of finite numbers of open straight segments 
   $\{ A_i \}, \{ B_j \}$ with the property that no two segments
   $A_i$ and no two segments $B_j$ have a common endpoint.
    Let the remaining portion of $\partial D$ consist of 
a finite number of open arcs $\{ C_k \}$, and of endpoints
   of $A_i$, $B_j$, and $C_k$.
    Let $\cal P$ denote a simple closed polygon whose vertices 
   are chosen from among the endpoints of the $A_i$ and $B_j$.
    Let 
   \begin{eqnarray*}
	\alpha =  \sum_{A_i\subset\cal P}length(A_i) \; , \; \; \; \; 
	\beta  = \sum_{B_j\subset\cal P}length(B_j) \; , \; \; \; \; 
	\gamma = \mbox{length of perimeter of }{\cal P} \; .
   \end{eqnarray*}
    Then if $\{ C_k \} \ne \emptyset$, there exists a solution 
   of the minimal surface equation in $D$ which assumes the 
   value $+\infty$ on each $A_i$, $-\infty$ on each $B_j$,
   and any assigned bounded continuous data on each open arc $C_k$
   if and only if
   \begin{displaymath}
	2\alpha < \gamma \mbox{   and   } 2\beta < \gamma
   \end{displaymath}
   for each polygon $\cal P$ chosen as above.
    Moreover, the solution is unique when it exists.
\end{theorem}

\begin{remark}				\label{rmk:wayne}
Note that in Theorem \ref{thm:Jen-Ser}, we allow the possibility that 
two different $C_k$ have a common endpoint.  We may have jump 
discontinuities in the boundary data at the points where two 
different $C_k$ meet.  It follows from the arguments in 
\cite{jes1} that, for $D$ as in Theorem \ref{thm:Jen-Ser}, 
if $u_1$ and $u_2$ are two solutions of the minimal surface equation 
such that $u_1 = u_2 = + \infty$ on each $A_i$ and 
$u_1 = u_2 = - \infty$ on each $B_j$ and $u_1 \ge u_2$ on 
each $C_k$, then $u_1 \ge u_2$ in the interior of $D$.  
\end{remark}

\setcounter{equation}{0}
\section{The Examples $\bf M_k$}			\label{sec:m_k}

 An immediate application of Theorems \ref{thm:krust} and 
\ref{thm:Jen-Ser} is to prove
that one can add more ends to Karcher's genus one 
surface $M_1$, thereby creating the surfaces $M_k$. 

\begin{theorem}					\label{thm:M_k}
    For each $k\ge 2$, there exists a one--parameter family
   $M_k$ of embedded, doubly--periodic minimal surfaces of 
   genus one with $4k$ Scherk--type ends.
\end{theorem}
\begin{proof}
    Fix $k$.
    The conjugate boundary of one eighth of $M_k$ is a
   graph over a rectangular domain with three sides at
   height zero and the fourth edge subdivided into $k$
   segments with heights alternating between $+\infty$
   and $-\infty$.
    Theorem \ref{thm:Jen-Ser} yields a Plateau solution
   with this boundary. Then Theorem \ref{thm:krust}, together with 
   Theorem \ref{thm:schwarz} and the
   maximum principle, gives the embedded surfaces $M_k$
   from these solutions.  
    The period problems associated to the ends, which equal
   the residues at the end punctures on the compact
   Riemann surface, are solved by 
   choosing the $A_i$ and $B_j$ to all be of the same length.
    Varying the length of the opposing zero height 
   sides of the rectangular domain yields a one-parameter 
   family of surfaces.
\end{proof}

 On the other hand, we immediately have: 
\begin{corollary}
   $M_{k}$ is a $k$--fold covering of $M_1$.
\end{corollary}
\begin{proof}
    Schwarz reflection (Theorem \ref{thm:schwarz}) 
   about line segments on the
   bounding conjugate contour for $M_1$ produces the bounding
   conjugate contour for $M_k$ for any $k$.
    The uniqueness of the minimal graphs in 
   Theorem \ref{thm:Jen-Ser} completes the proof.
\end{proof}

 We included these examples $M_k$ because the technique
used to solve the $k$--dimensional period problem
arising from the additional ends is used throughout 
the paper.
 In particular, we have
\begin{lemma}						\label{lem:mk}
    Each collection of surfaces, 
   $M_k^+$, $M_k^-$, $M_{2k+1}^{++}$, $M_{2k+1}^{--}$, 
   $M_{k}^{+-}$,
   results from adding ends and handles to $M_1$, and 
   the period problems arising from the additional ends 
   are all solved as above, i.e. by requiring 
   \begin{displaymath}
	\epsilon = \mbox{length}(A_i) = \mbox{length}(B_j)
   \end{displaymath}
   to be constant for all $i, j$.
\end{lemma}

A proof of this lemma is contained in the appendix of this paper.  

The observation that this restriction on the conjugate contour solves
all the periods arising from additional ends demonstrates that these
periods are independent of the periods arising from additional handles.
 This restriction enables us to eliminate all but one or 
two periods in these surfaces, so we may focus only 
on the periods arising from the new handles.

\setcounter{equation}{0}
\section{The Examples $\bf M_k^+$}			\label{sec:m_k+}
\begin{figure}
\begin{center}
  \includegraphics[width=7.5cm]{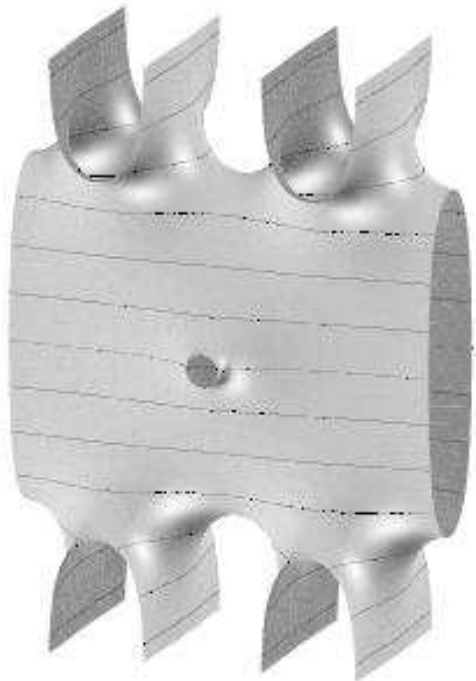} 
\end{center}
  \caption{Fundamental piece of $M_2^+$.}           \label{fig:2k+_image}
\end{figure}

\begin{figure}
\begin{center}
  \includegraphics[width=8.5cm]{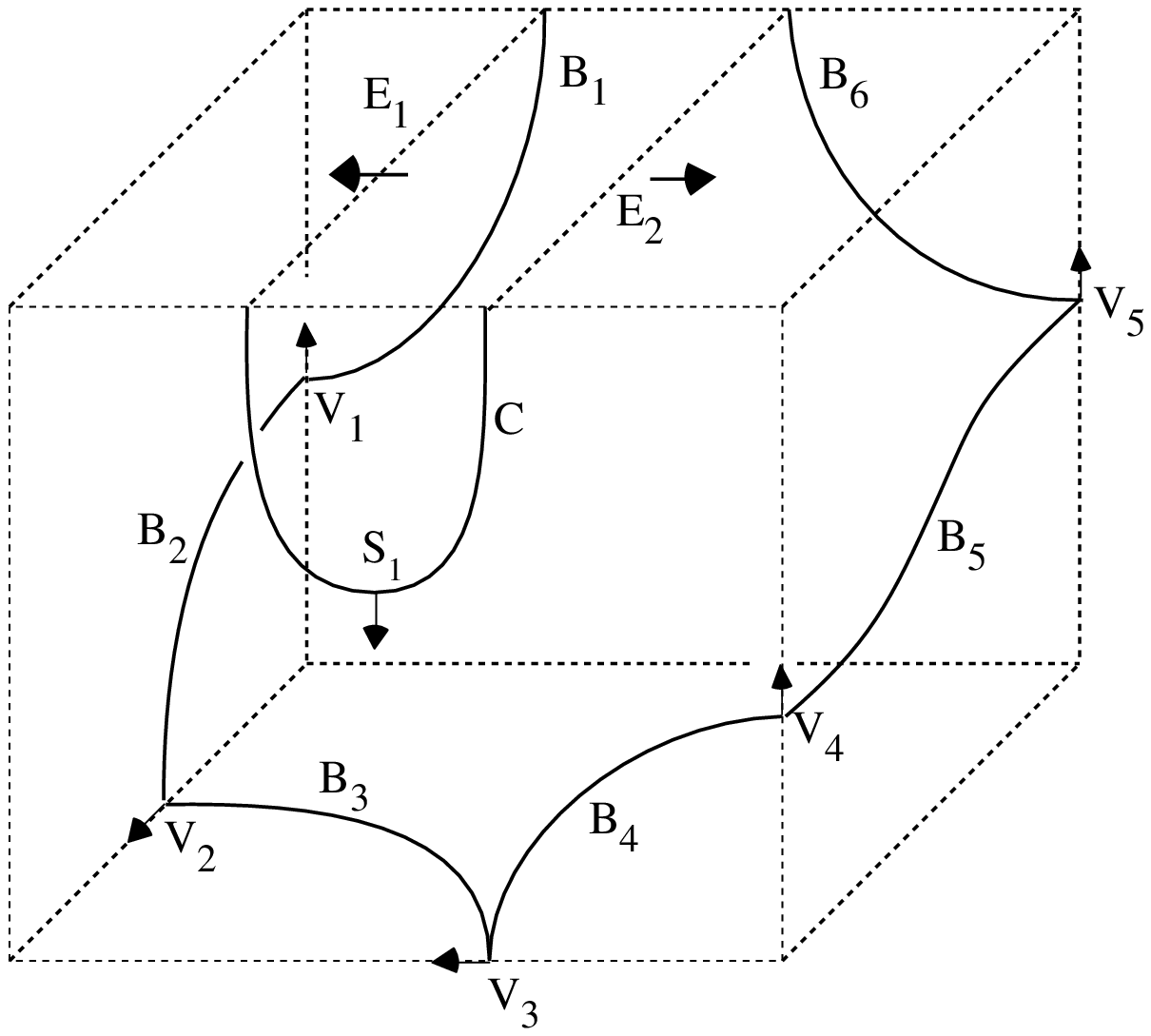} 
  \includegraphics[width=8.5cm]{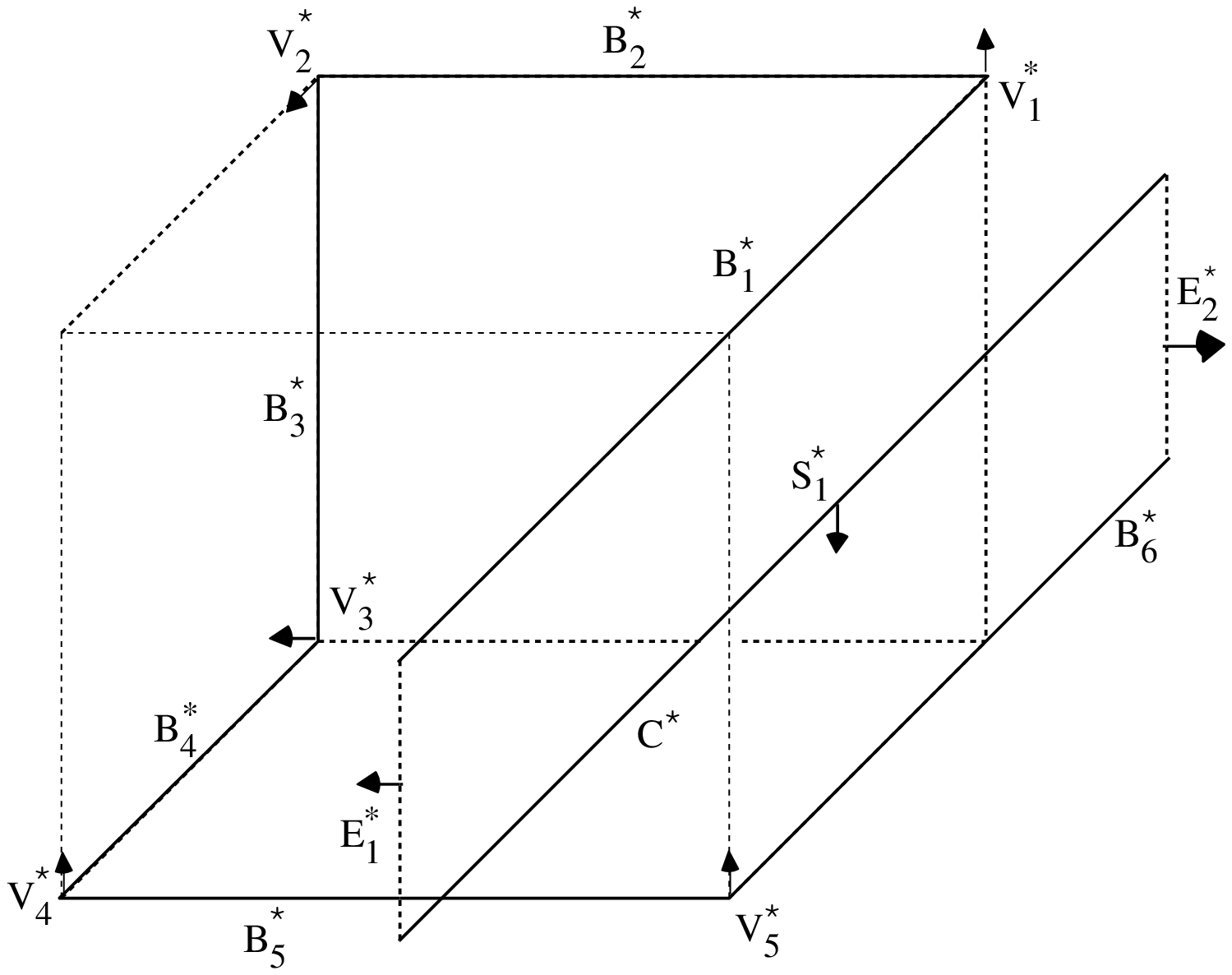} 
  \includegraphics[width=8.5cm]{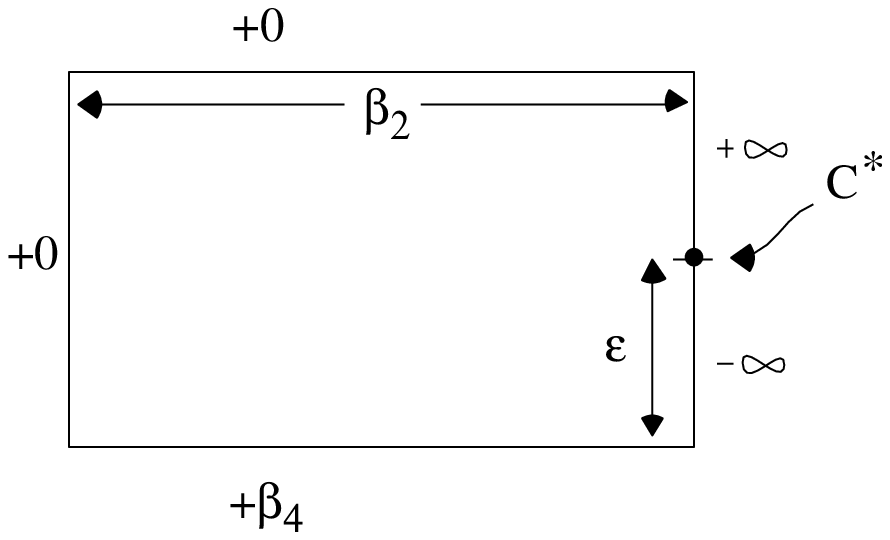} 
\end{center}
   \caption{Sketches of the boundaries of one eighth of
		$M_2^+$ and its conjugate (top row), and
		the graph dimensions and heights over the front 
		face of the bounding box for the 
		conjugate contour (bottom).}            \label{fig:2k+_sketches}
   \hfill
\end{figure}

 The sketch in Figure \ref{fig:2k+_sketches}, upper--left 
is of a contour suggestive 
of a `+' type handle in an even ended surface
which we will use to motivate the discussion.
 Taking its conjugate contour produces the contour 
in Figure \ref{fig:2k+_sketches}, upper--right, which is bounded 
by line segments as labelled in the figure.
 This contour bounds a Jenkins-Serrin graph over the 
front face of the box and hence supports a solution
to the Plateau problem.
 Let $\beta_j = \mbox{Length}(B_j) = \mbox{Length}(B_j^*)$ 
for $j=2,3,4,5$.
 The symmetries of $M_k^+$ imply there are $k$ periods, 
$k-1$ of these resulting from the ends, and one arising 
from the new handle.
 Lemma \ref{lem:mk} implies that if we restrict the 
conjugate contours so that the lengths of the segments 
over which the boundary contour is unbounded are equal, 
then $k-1$ of these periods are zero.
 Let $\epsilon = \beta_3/k$ be this common length.

 The remaining period
is shown to change sign as $\beta_4$ is varied so that the 
intermediate value theorem implies the following
\begin{theorem}
    For each $k>0$, there exists a one-parameter family
   $M_{k}^{+}$ of embedded, doubly--periodic  
   minimal surfaces of genus two with $4k$ Scherk type
   ends.
\end{theorem}

We give the argument only for the case $k=2$, as the argument is 
essentially identical for all $k$.  
 Choosing the curves $B_2^*$ and $B_3^*$ to lie 
at the zero level, the height of $B_5^*$ is $+\beta_4$,
with the end $E_1^*$ at height $+\infty$ and
$E_2^*$ at $-\infty$ as indicated in 
Figure \ref{fig:2k+_sketches}, bottom.

\begin{proof}
    All that remains to be shown is that as $\beta_4$ is 
   varied, the period 
	$\pi(\beta_4) = \mbox{Re} \int_{S_1}^{V_4} \phi_2$
   changes signs.
    Note that this period measures the distance between 
   the planes containing the curves $B_4$ and $C$.

    Let $\beta_2 > \epsilon$ and
   consider the two cases of $\beta_4$ large and $\beta_4$ small:

    a) Let $\beta_4 \to 0$. The limiting 
   surface is $M_2$ and the embeddedness of $M_2$ implies
   the point $V_4$ lies behind the symmetry
   plane of $C$; so $\pi(\beta_4) < 0$ for $\beta_4$ near zero.

    b) For large $\beta_4$, we claim the distance between the planes 
   containing $B_4$ and $B_6$ is $\beta_2 - \delta > \epsilon$, 
   since the \Gau\ map approaches a constant along $B_5$.
    To see this, use the barrier surface given as a Jenkins-Serrin
   graph over the back face of the box in Figure 
   \ref{fig:2k+_sketches}, right, with height $+\infty$ over the
   edge $B_5^*$ and the same heights as the contour for $M_2^+$
   over all other edges.  Arguments in \cite{jes1} imply the conjugate 
graphs converge to the barrier surface as $\beta_4 \to \infty$.  
    So in the limit, the behavior of the ends is the same and
   therefore $B_5$ approaches a straight line of length $\beta_2$
   which is greater than $\epsilon$.
    Hence $B_4$ lies in front of $C$ and $\pi(\beta_4) > 0$ for 
   $\beta_4$ large.

    Hence the period problem is solvable.
    Since $\beta_2$ is only bounded below this shows the period problem
   can be solved for each $\beta_2>\epsilon$ and therefore there exists a 
   one--parameter family of these surfaces.
Theorem \ref{thm:krust} implies each eighth of any one of these surfaces
   is embedded and, by Theorem \ref{thm:schwarz}, extends to an embedded 
   minimal surface.
\end{proof}

\noindent
{\bf \Weierstra\ data for $M_k^+$:} 
 Since $M_k^+$ is invariant under an order--two 
normal symmetry about
the $x_3$-axis, with six fixed points, the quotient is a sphere
minus $2k$ points.
 The meromorphic function $g^2$, where $g$ is the \Gau\
map, descends to the quotient.
 Taking $z$ to be the coordinate on this sphere, we normalize so that
$z(V_3) = 0$, $z(V_2) = \infty$ and $e_k = z(E_k) = 1$.
Define $v_i$ by $v_i = z(V_i)$ for $i=1,4,5$, 
$s_j = z(S_j)$, $j = 1, 2, \ldots, k-1$, where $\{S_j\}$ are
the vertical points lying on the planar geodesics between the
ends,
$e_m = z(E_m)$, $m = 1, \ldots, k-1$.  
 Conformality of $z$ orders these values 
$0<v_4<v_5<1<s_{k-1}< e_{k-1}<s_{k-2}<\ldots<s_1<e_1<v_1$.

 Comparison of the meromorphic functions $g^2$ and $z$ leads to 
the following \Weierstra\ data for $M_k^+$:
\begin{equation}
   \begin{array}{rcl}
      g^2  &=& \ds\frac{z+v_4}{z-v_4}\; \frac{z+v_5}{z-v_5} \;
		\frac{z+(-1)^k v_1}{z-(-1)^k v_1} \cdot
		f_k^2(z, s_1, \ldots, s_{k-1}), \\
	& & \\
      \eta &=& \ds \frac{dz}{{\cal E}_k(z,e_1, \ldots, e_k)} 
		\; D_k(z, s_1, \ldots, s_{k-1})
		\; N_k(z, s_1, \ldots, s_{k-1}),
   \end{array}						\label{eq:2k+_wd}
\end{equation}
where 
$f_k(z, s_1, \ldots, s_{k-1}) = N_k(z, s_1, \ldots, s_{k-1}) / 
D_k(z, s_1, \ldots, s_{k-1})$, with
\begin{displaymath}
   \begin{array}{rcl}
      N_k(z,s_1,\ldots,s_{k-1}) & := & 
		\ds \prod_{j=1}^{k-1} (z + (-1)^{k+j} s_j),\\
      D_k(z,s_1,\ldots,s_{k-1}) & := & 
		\ds \prod_{j=1}^{k-1} (z - (-1)^{k+j} s_j),\\
      {\cal E}_k(z,e_1,\ldots,e_k) &:=&
		\ds \prod_{m=1}^k (z^2 - e_m^2). 
   \end{array}
\end{displaymath}
 The conditions for embedded ends are
\begin{equation}
   g^2(1) = g^2(e_m) = 1					\label{eq:2k+_constr}
\end{equation}
for all $m\le k$.
 For $k=2$, we have the constraints
\begin{eqnarray*}
   A(1+v_4)(1+v_5) &=& \tilde A (1-v_4)(1-v_5) \\
   B(e_2+v_4)(e_2+v_5) &=& \tilde B (e_2-v_4)(e_2-v_5), 
\end{eqnarray*}
where $A = (1-s_1)^2(1+v_1)$, $\tilde A = (1+s_1)^2(1-v_1)$,
$B = (e_1-s_1)^2(e_1+v_1)$, and $\tilde B = (e_1+s_1)^2(e_1-v_1)$.
 From this, we can derive  the conditions
\begin{eqnarray}
   v_4 v_5   &=& {{ (A-\tilde A )(B+\tilde B )e_1 - 
	(A+\tilde A )(B-\tilde B )e_1^2}
	\over{ (A+\tilde A )(B-\tilde B ) - 
	(A-\tilde A )(B+\tilde B )e_1}},      \label{eq:2k+_c1} \\
   v_4 + v_5 &=& {{ (A-\tilde A )(B+\tilde B )(e_1^2-1)}
	\over{ (A+\tilde A )(B-\tilde B ) - 
	(A-\tilde A )(B+\tilde B )e_1}}.	\label{eq:2k+_c2}
\end{eqnarray}
And $v_4$ and $v_5$ are the zeros of a degree--two polynomial.

 With the \Weierstra\ data (\ref{eq:2k+_wd}) and the constraints
(\ref{eq:2k+_c1}) and (\ref{eq:2k+_c2}) 
we get the image in Figure \ref{fig:2k+_image} after choosing 
$k=2$ and determining
the correct values for $v_1$ and $e_1$.

\setcounter{equation}{0}
\section{The Examples $\bf M_k^-$}			\label{sec:m_k-}

 The periods associated to $M_k^-$ arise as residues
at the punctures for the ends or from integrating 
along a curve representing a non-trivial homotopy class.
 As in the case of $M_k$, the residues at the ends 
are made equal by equally distributing the straight
lines lieing between the ends of the conjugate of one-eighth of the 
fundamental piece.
 The portion of the period problem resulting from non-trivial
homotopy classes is one-dimensional due to the symmetries
of $M_k^-$, and use of the same barriers as in \cite{wei2}
shows that this period is also solvable.
 Hence 
\begin{theorem}
    There exists a one--parameter family of genus two,
   embedded minimal surfaces $M_k^-$ with $4k$ Scherk--type
   ends, for all $k \ge 1$.
\end{theorem}

\setcounter{equation}{0}
\section{The Examples $\bf M_{2k+1}^{--}$}		\label{sec:m_2k+1--}

\begin{figure}
\begin{center}
  \includegraphics[width=6.5cm]{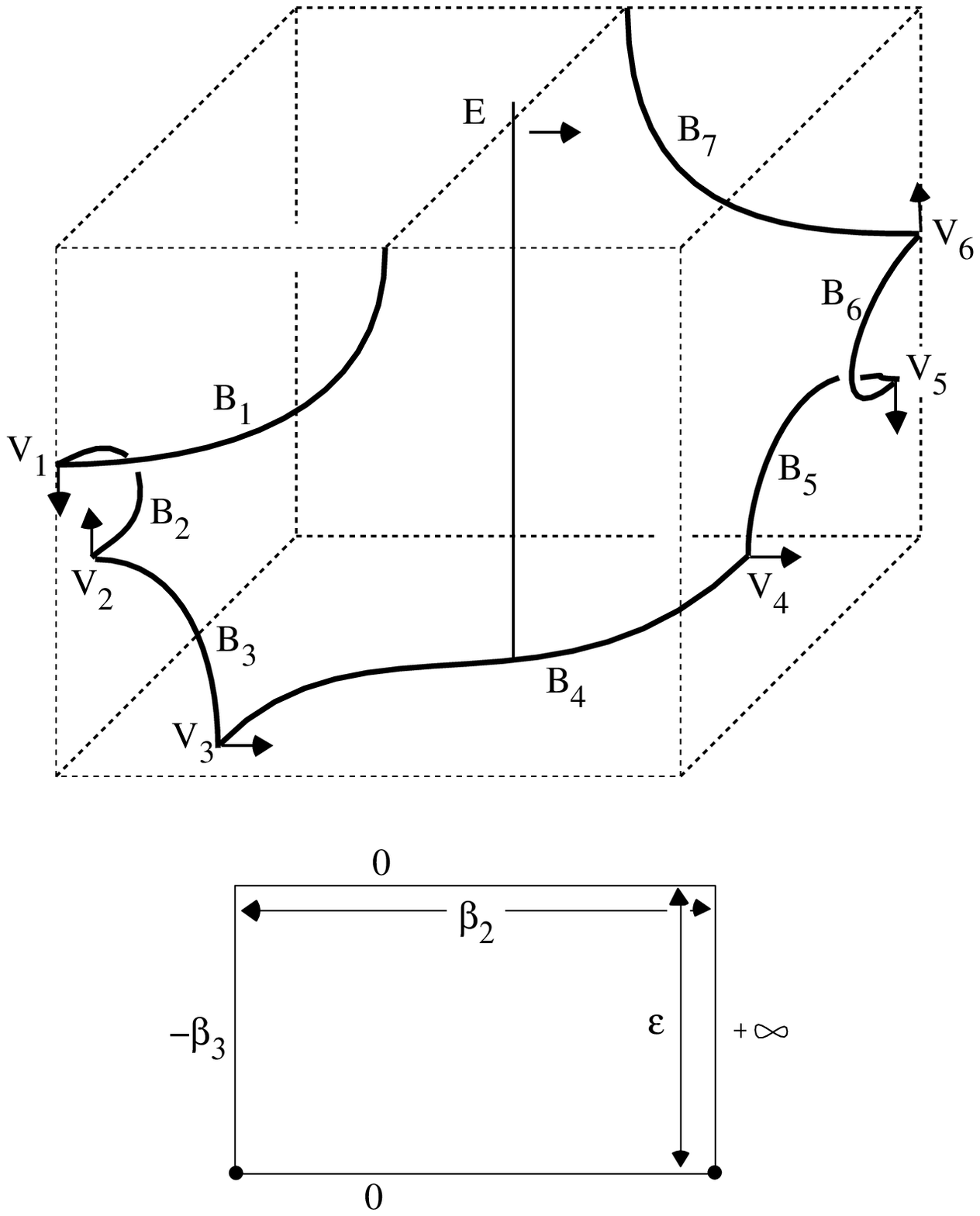} 
  \includegraphics[width=6.5cm]{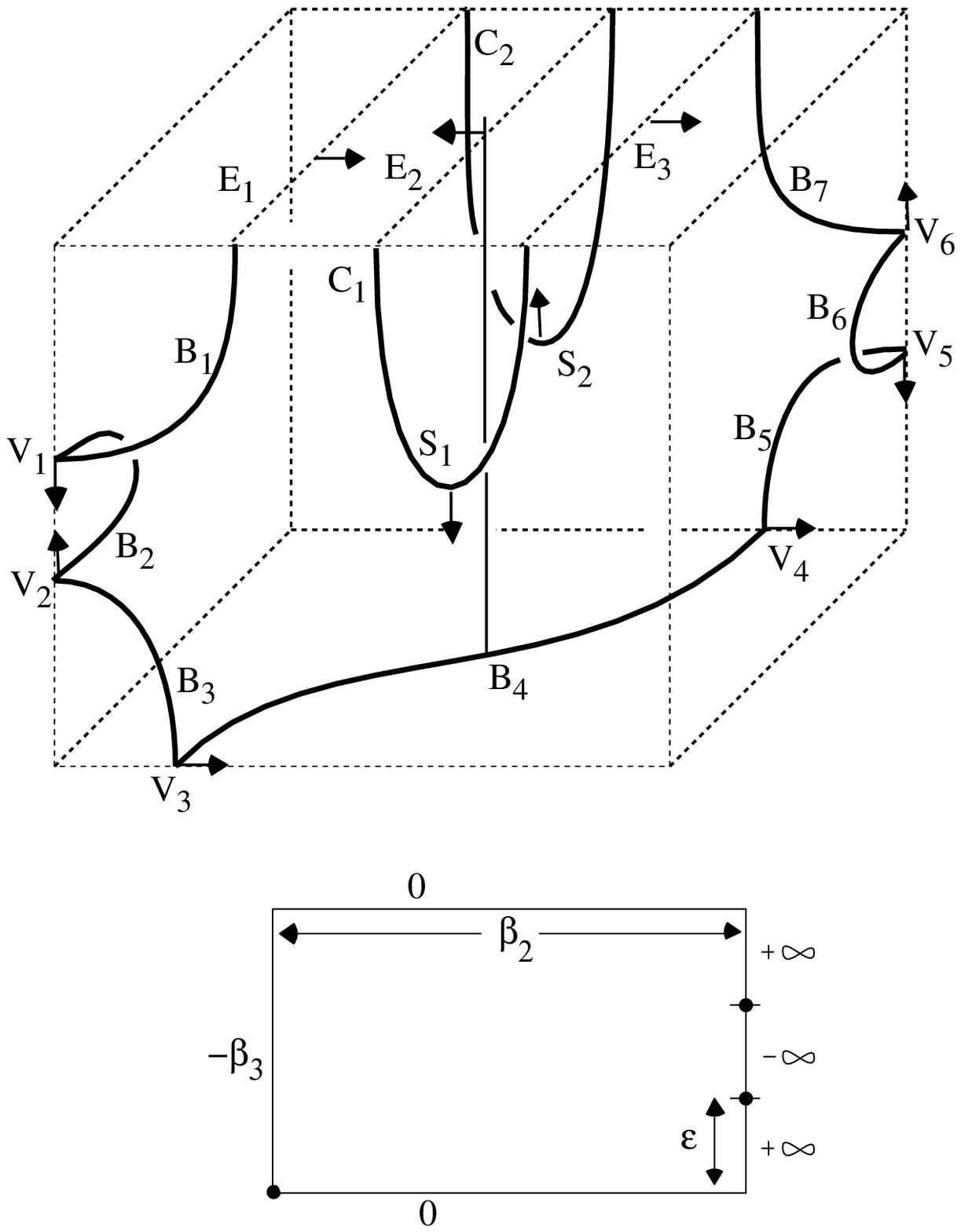} 
\end{center}
   \caption{Sketches of the boundary of one eighth of
	$M_1^{--}$ (upper left) and $M_3^{--}$ (upper right) with 
	the parameters for the conjugate boundary contour 
	viewed as a graph
	over the rectangular region drawn below 
	each sketch.}			\label{fig:2k+1mm_sketches}
\end{figure}

 In this section, we construct the embedded minimal surfaces
$M_{2k+1}^{--}$.  
 Specifically, we construct genus three surfaces having 
all the symmetries of Karcher's genus one surface $M_1$,
with two `--' handles and $4(2k+1)$ Scherk ends.

 F. Wei modified $M_1$ by introducing a single handle
over one of its saddle points.
 In the sketches of Figure \ref{fig:intro_sketches}, 
one can see that this results in a new vertical point
over $V_4$.
 In order to retain the vertical straight lines of $M_1$
on higher genus surfaces, one
is obliged to add a handle over the other 
saddle point, since, by Theorem \ref{thm:schwarz}, 
$180^o$ rotation about these
straight lines are isometries of the surface.
 Such a surface might have a boundary like that 
sketched in Figure \ref{fig:2k+1mm_sketches}, left.
 If this surface did exist, it's conjugate contour 
would be as in Figure 
\ref{fig:2k+1mm_sketches}, lower left.  
This conjugate contour meets all the conditions of 
Theorem \ref{thm:Jen-Ser}, hence it supports a solution to the 
Plateau Problem, and the original surface conjugate to this 
solution is a minimal surface bounded by planar curves 
with the desired symmetries.

 Although the conjugate surface is 
a minimal surface bounded by planar curves,
it is not guaranteed that 
reflection in these planes produces an embedded 
doubly--periodic surface.
 In particular, using the notation of Figure
\ref{fig:2k+1mm_sketches}, one does not know
if the curves $B_1$ and $B_3$ lie in the same plane.
 This brings us to the period problem; 
one must insure that the planes containing $B_1$ and $B_3$ coincide.
 Since we have assumed the surface contains a 
vertical straight line,
knowing $B_1$ and $B_3$ lie in 
the same plane implies the planes containing
$B_5$ and $B_7$ also coincide.
 Should this period problem be solvable, the surface
in our notation would be denoted by $M_1^{--}$.

 In Theorem \ref{thm:2k+1mm}.2 we prove, by analyzing
the Plateau solutions for the countour of Figure 
\ref{fig:2k+1mm_sketches}, lower left, that this period 
problem can never be solved. 
 In contrast, by having
more ends on the surface, as in Figure 
\ref{fig:2k+1mm_sketches}, right, we prove in Theorem
\ref{thm:2k+1mm}.1 that the obstruction to solving 
this period problem is removed.
 These new surfaces are the surfaces $M_{2k+1}^{--}$ 
in our notation.

\begin{theorem}					\label{thm:2k+1mm} 
   \setcounter{i}{1}
   \begin{list}{\em (\arabic{i})}{\usecounter{i}}
	\item For each $k\ge1$, there exists a 
		one--parameter family 
		of embedded, doubly--periodic minimal 
		surfaces $M_{2k+1}^{--}$ of genus 
		three;

	\item $M_1^{--}$ does {\bf not} exist.
   \end{list}
\end{theorem}

\begin{proof} 
    Let $\beta_j = \mbox{Length}(B_j) = \mbox{Length}(B_j^*)$ 
   for $j=2,3,4,5,6$.
    By Lemma \ref{lem:mk}, all periods arising
   from the addition of ends are zero provided the 
   lengths of the segments over which the conjugate contours 
   are unbounded are equal.
    We assume this condition, and 
   let $\epsilon$ be this common length, which remains fixed
   throughout the proof.
    Hence we need only address the periods arising from 
   non-trivial homotopy classes, i.e. from the addition of
   new handles.
    Note that from the conjugate contour one sees that 
   for each $M_{2k+1}^{--}$, $\beta_4 = (2k+1)\epsilon$.

    {\bf Proof of (2):}
    We proceed by contradiction.  
    Suppose $M_1^{--}$ does exist.
    Let $S$ be one eighth of $M_{1}^{--}$.  
    Figure \ref{fig:2k+1mm_sketches}, left shows a sketch of $S$. 
    We are assuming that there is a vertical 
   straight line on $S$ passing
   through the end $E$, orthogonal to the plane containing $B_4$.
    Rotation about this line interchanges $V_1$ and $V_6$, and 
interchanges $V_2$ and $V_5$.  

   \begin{remark}
       The boundary contour of $S^*$ is a graph
      over a rectangle as drawn in Figure 
      \ref{fig:2k+1mm_sketches}, lower left, and as a result
      of the symmetries, $B_2^*$ and $B_6^*$ lie
      at the same height.
       Choosing this to be the zero height implies the
      line $B_4^*$ has height $-\infty < -\beta_3 < 0$, 
      and the end $E$ has height $+\infty$.
       From Theorem \ref{thm:Jen-Ser}, we get a 
      minimal graph with this boundary.
       As a graph, it is embedded and 
      Theorem \ref{thm:krust} assures that $S$ is embedded.
       Hence there exists a Plateau solution 
      $S^*$ with the desired boundary and symmetries. 
   \end{remark}
   \noindent 
   {\bf Claim:}{ \em
       The distance between the planes containing 
      $B_3$ and $B_5$ is always shorter
      than the distance between the planes containing
      $B_1$ and $B_7$.
       Hence the period is always of the same sign.
   }

    The planar geodesic $B_4$ has length $\epsilon$ and is 
   not a straight line.  
    Therefore the distance between the symmetry
   planes containing $B_3$ and $B_5$ is strictly less than 
   $\beta_4 = \epsilon$, 
   and the curve $B_3$ always lies to one side 
   of the plane containing $B_1$.
    This establishes the claim and completes the 
   proof of (2).

 In summary, the period problem on $M_1^{--}$ 
is unsolvable because the distance $\epsilon$ between 
the planar curves bounding the end is equal to
$\beta_4$ and the planar curve $B_4$ is not straight.
 If one could modify the conjugate contour so that 
$\beta_4 > \epsilon$,
then the period problem may be solvable.
 One way of achieving this is to add more ends to the
conjugate contour as in the sketch in Figure 
\ref{fig:2k+1mm_sketches}, lower right.  
 Because we wish to maintain the vertical straight lines,
the contour bounded by straight lines must have a 
horizontal planar symmetry.
 Therefore we must add an even number of extra ends.
 Figure \ref{fig:2k+1mm_sketches}, right is a sketch of 
such a surface with three ends.
 The conjugate contour for this surface 
is again a Jenkins-Serrin graph over a rectangle as drawn
in Figure \ref{fig:2k+1mm_sketches}, lower right.

{\bf Proof of 1):}
    Assume $\beta_2 > \beta_4 = (2k+1) \epsilon$.  
     Since we have assumed the existence of a 
   vertical straight line on the 
   surface passing thru $E_{k+1}$ and orthogonal to $B_4$, 
   we have only one period arising from a non-trivial 
   homotopy class.
    For this period, we must show that $B_1$ lies in the plane
   containing $B_3$.  We use the intermediate 
   value theorem to show the existence of a value 
   for $\beta_3$ such that this period is zero.
    Specifically we have two cases:

    a) As $\beta_3 \goesto 0$, $M_{2k+1}^{--}$ degenerates to 
   $M_{2k+1}$. 
    By the embeddedness of $M_{2k+1}$, we have the point
   $V_2$ moves behind the plane containing  $B_1$,
   and the period is negative.

    b) As $\beta_3 \to \infty$, the curve $B_4^*$ moves away
   toward height $-\infty$.  
    Let $\cal{B}$ be the Jenkins--Serrin graph over the 
   rectangle as described in Figure 
   \ref{fig:2k+1mm_sketches}, lower right, 
   with boundary heights $0, -\infty, 0, +\infty, -\infty, +\infty$
   with zero heights corresponding to the edges containing
   the curves $B_2^*$ and $B_6^*$.  This graph $\cal B$ exists, since 
$\beta_2 > \beta_4$, by Theorem \ref{thm:Jen-Ser}.  
By the arguments in \cite{jes1}, the conjugate graphs converge to 
$\cal B$ as $\beta_3 \to \infty$.  
    Therefore along $B_4^*$ the \Gau\ map 
   approaches a constant value, and the 
   displacement along $B_4$ in the desired
   direction approaches $(2k+1) \epsilon = \beta_4$. 
    Hence $V_2$ lies in front of the plane containing $B_1$ for 
large $\beta_3$, and the period is positive.  

    By the intermediate value theorem, there exists a
   value of $\beta_3$ at which the period is zero.
    Therefore the period problem is solvable on 
   $M_{2k+1}^{--}$.

Theorems \ref{thm:Jen-Ser} and \ref{thm:krust} imply that the one 
eighth portion $S$ of $M_{2k+1}^{--}$ is embedded.  Applying the 
classical maximum principle and the maximum principle at infinity
\cite{mr1}, one easily determines that the full fundamental piece of 
$M_{2k+1}^{--}$ lies inside the box given by its boundary curves.
    Reflections through the faces of this box 
   produces an embedded surface.
    Therefore $M_{2k+1}^{--}$ is embedded.

    $\beta_2$ has not been used in 
   this argument ($\beta_2$ is any fixed number greater than 
$\beta_4$), and therefore we have a one-parameter 
   family of $M_{2k+1}^{--}$ for each $k \geq 1$.
\end{proof}

\vspace{0.1in}
\noindent
{\bf \Weierstra\ data for $M_3^{--}$:}
 Since $M_{3}^{--}$ is invariant under an order--two normal
symmetry about the $x_3$-axis, with eight fixed points, the
quotient surface is a sphere.
 The meromorphic function
$g^2$, where $g$ is the stereographic projection of
the \Gau\ map, descends to the quotient.
 Taking $z$ to be the coordinate on the sphere, we normalize 
so that $z(V_3) = \infty$, $z(V_4) = 0$, and $z(E_2) = 1$.
 With this normalization, rotation about the vertical 
straight line on $M_3^{--}$ corresponds
to inversion through the unit circle.
 Define $e_1 = z(E_1), v_j = z(V_j)$, for $j=1,2$, 
and $s_1 = z(S_1)$.
 Then $z(E_3) = 1/e_1, z(V_5) = 1/v_2, z(V_6) = 1/v_1$,
and $z(S_2) = 1/s_1$.
 Comparison of the meromorphic functions $g^2$ and $z$ leads to 
the following \Weierstra\ data for $M_3^{--}$:
\begin{equation}
   \begin{array}{rcl}
      g^2 &=& \ds \frac{z-v_1}{z+v_1} \ \ds \frac{z+1/v_1}{z-1/v_1} \
		\ds \frac{z+v_2}{z-v_2} \ \ds \frac{z-1/v_2}{z+1/v_2} \
		\left( \ds \frac{z-s_1}{z+s_1} \right)^2
		\ \left( \ds \frac{z+1/s_1}{z-1/s_1} \right)^2, \\
	& & \\
      \eta &=& \ds \frac{dz}{z^2-1} \ \ds \frac{z^2-s_1^2}{z^2 - e_1^2} \ 
	\ds \frac{z^2 - 1/s_1^2}{z^2 - 1/e_1^2}.
   \end{array}							\label{eq:2k+1mm_wd}
\end{equation}
 This \Weierstra\ data insures each Scherk-type end is
itself an embedded end, but one must also guarantee that the 
limit normals on the ends are antipodal so the ends do not
cross each other as they diverge.
 Because of our choice of orientation, this is equivalent to
the conditions
\begin{displaymath}
   g^2(1) = g^2(e_1) = g^2(1/e_1) = 1.
\end{displaymath}
 Due to the rotational symmetry, 
the second and third conditions result in the same constraints,
while the first is automatically satisfied.
 The second condition places the following constraint on $e_1$:
\begin{eqnarray}					\label{eq:2k+1mm_b}
   (\delta + \gamma + 2\nu) e_1^6 
	& + & \left[ (\delta+\gamma)(\nu^2-1) + 2(\delta\gamma-2)\nu 
		- (\delta+\gamma) - 2\nu \right] e_1^4 \nonumber \\
	& + & \left[ 2\nu - (\delta+\gamma)(\nu^2-2)-2\nu(\delta\gamma-2)
		+ (\delta+\gamma) \right] e_1^2 \\
	& - & 2\nu - (\delta+\gamma) = 0, \nonumber
\end{eqnarray}
where $\delta = v_2 - 1/v_2$, $\gamma = 1/v_1 - v_1$, 
and $\nu = 1/s_1 - s_1$.

 By Theorem \ref{thm:2k+1mm}, there exists a solution to 
(\ref{eq:2k+1mm_b}) in the necessary range.
 Using the computer to find this solution and to calculate the
values of the two periods of the \Weierstra\ data, we determine
the appropriate values for $e_1$, $v_1$, and $v_2$, given a value for 
$s_1$.  We thereby generate the image of $M_3^{--}$ 
in Figure \ref{fig:2k+1pp_image}.

\setcounter{equation}{0}
\section{The Examples $\bf M_{2k+1}^{++}$}		\label{sec:m_2k+1++}

\begin{figure}
\begin{center}
  \includegraphics[width=10.5cm]{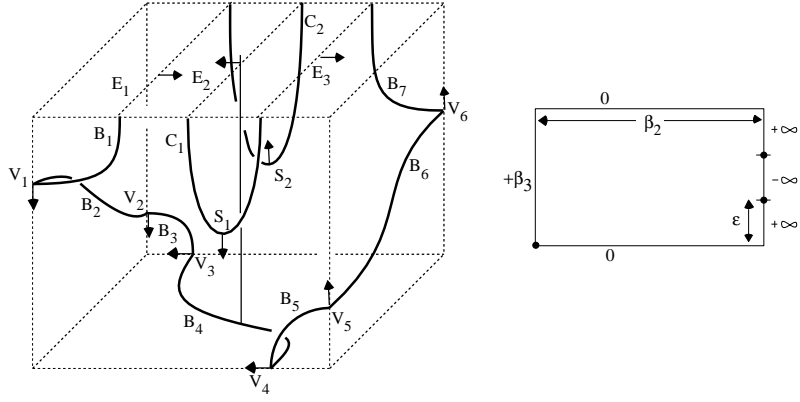} 
\end{center}
   \caption{Sketches of the boundary of one eighth of $M_3^{++}$ 
	and it's conjugate boundary graph heights over the front
	face of the bounding box.}               \label{fig:2k+1pp_sketch3}
\end{figure}

 As in the previous section, one might investigate whether it
is possible to construct genus three examples by adding two 
`+' type handles to $M_1$ while preserving the symmetries.
 The same methods as those used in the `--' case 
can be used to show
the existence of a minimal disc with the desired boundary
and symmetries, but one must again consider the 
period problem.
 The similarities between the conjugate contours for the
two `--' handles and two `+' handles allow one to 
observe a similar natural obstruction to solving
the period problem for the one-ended surfaces.
 By adding more ends to these surfaces, as in the
previous section,
this obstruction is overcome.
 Denoting these new surfaces by $M_{2k+1}^{++}$
and using arguments similar to those used in 
the proof of Theorem \ref{thm:2k+1mm}, one has: 
\begin{theorem}
   \setcounter{i}{1}
   \begin{list}{\em \arabic{i})}{\usecounter{i}}
      \item There exists a one--parameter family of embedded, 
		doubly--periodic minimal surfaces 
		$M_{2k+1}^{++}$ of genus three, for each $k\ge 1$.

      \item $M_1^{++}$ does {\bf not} exist.
   \end{list}
\end{theorem}

{\bf Note:} The symmetry groups for $M_1$, $M_{2k+1}^{--}$,
and $M_{2k+1}^{++}$ are identical.
 Hence one has two collections of genus three minimal surfaces
with the same symmetries as Karcher's original genus one surface 
$M_1$.  

\noindent
{\bf \Weierstra\ data for $M_3^{++}$:}
 Using the same notation as that used for the surface $M_3^{--}$,
we can determine the \Weierstra\ data for $M_3^{++}$;
the results are as follows: 
\begin{eqnarray*}
   g^2 &=& 	\frac{z-v_1}{z+v_1} \ \frac{z+1/v_1}{z-1/v_1} \
	 \frac{z-v_2}{z+v_2} \ \frac{z+1/v_1}{z-1/v_2} \
	\left( \frac{z-s_1}{z+s_1} \right)^2
	\ \left( \frac{z+1/s_1}{z-1/s_1} \right)^2, \\
	& & \\
   \eta &=& \frac{dz}{z^2-1} \ \frac{z^2-s_1^2}{z^2 - e_1^2} \ 
	\frac{z^2 - 1/s_1^2}{z^2 - 1/e_1^2}.
\end{eqnarray*}
 With the same constraints for parallel ends as in 
(\ref{eq:2k+1mm_b}) and by changing $\gamma$ to 
$v_1 - 1/v_1$ we compute the parameters used in 
generating the image in Figure \ref{fig:2k+1pp_image}.

\setcounter{equation}{0}
\section{The Examples $\bf M_{k}^{+-}$}			\label{sec:m_k+-}

 In this section, we consider the genus three 
surfaces $M_{k}^{+-}$ which arise by adding both a 
'$+$' handle and a '$-$' handle to $M_k$.
 As in the case of $M_k^-$ and $M_k^+$, the handles make
it impossible to preserve the straight line symmetries
of $M_1$, but the three mutually perpendicular
planar reflectional symmetries are preserved.
 These symmetries reduce the number of periods 
that need to be addressed in order for the period problem to
be solved.
 In particular, $M_{k}^{+-}$ has $k+1$ periods:  $k-1$ 
of these periods arise from the residues of the Weierstrass 
data at the ends; leaving only two periods resulting from 
non-trivial homotopy classes.
 By Lemma \ref{lem:mk}, the periods resulting
from the additional ends are simultaneously zero 
provided the segments over which the conjugate contours
are unbounded are equal in length.
 As we have done in the previous sections, we fix $\epsilon$
to be this common length.
 Now we need only consider the two periods that result 
from non-trivial homotopy classes.

 Figure \ref{fig:mkpm_sketches} contains sketches of the
boundary of one eighth of $M_2^{+-}$ (left) and $M_3^{+-}$
(right), together with the conjugate contour heights 
written as a graph, where
$\beta_j = \mbox{Length}(B_j)$ for $j=2,3,4,5,6$.
 We assume $\beta_2 > \epsilon$ on all contours.

We now consider the case $k=1$.  
 By consideration of the two periods for 
$M_{1}^{+-}$ for varying values of $\beta_3$ and $\beta_5$, 
we are able to use a two--dimensional degree argument to
prove: 
\begin{theorem}					\label{thm:mkpm}
    There exists a one-parameter family of genus three,
   embedded minimal surfaces $M_{1}^{+-}$ with $4$ Scherk-type
   ends.
\end{theorem}

\begin{figure}
\begin{center}
  \includegraphics[width=5.5cm]{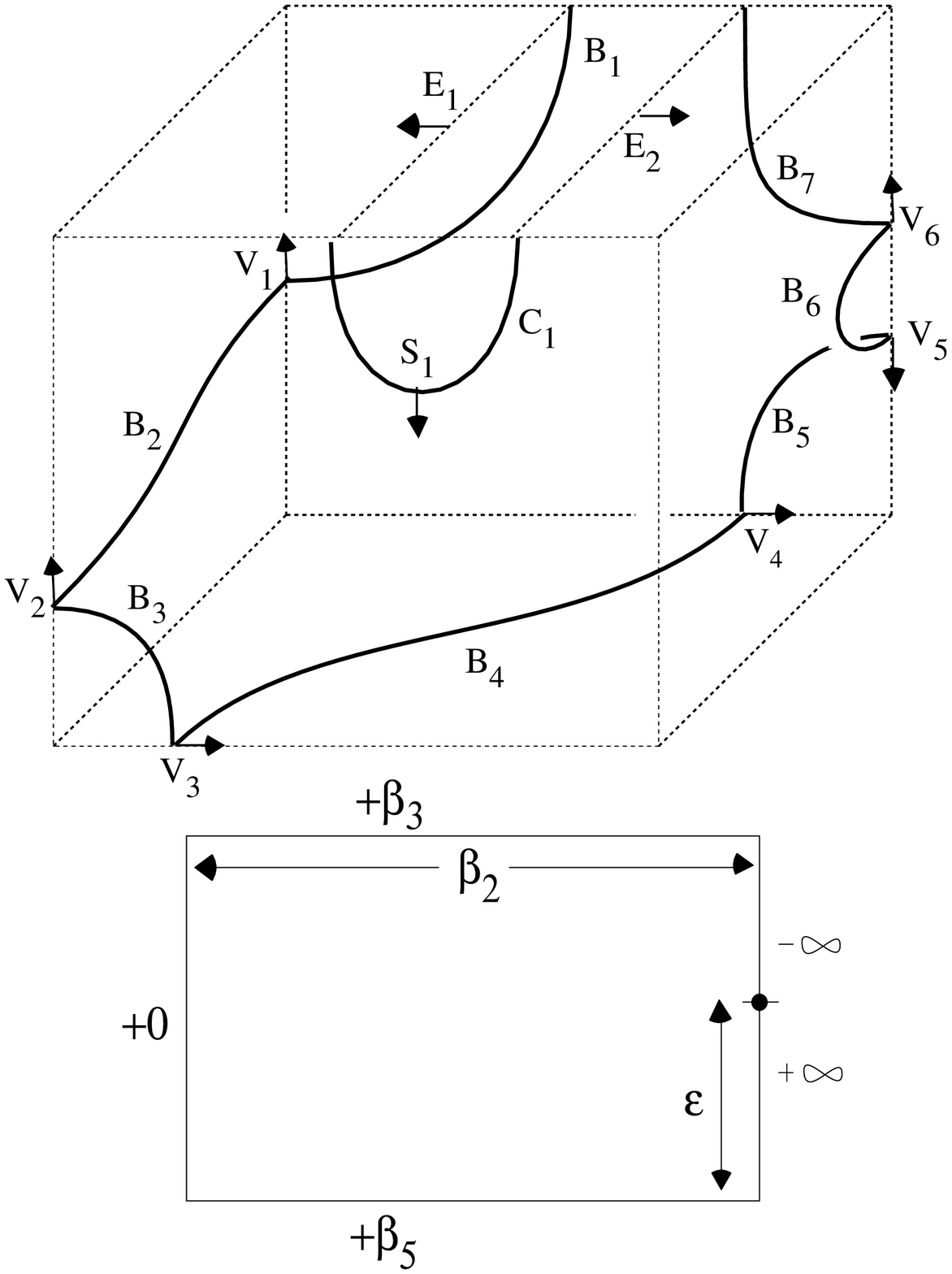} 
  \includegraphics[width=5.5cm]{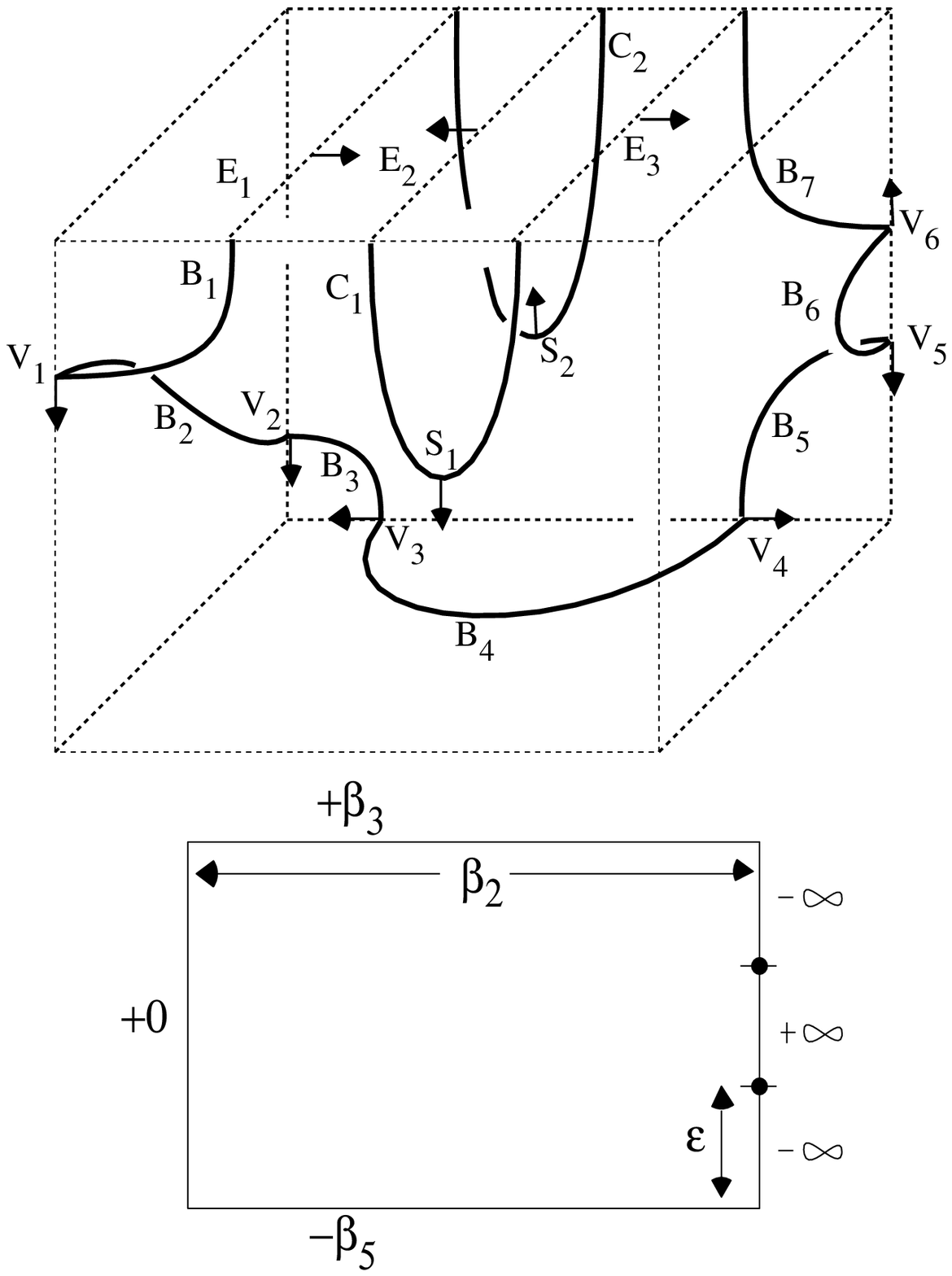} 
\end{center}
   \caption{Sketches of one eighth of $M_{2}^{+-}$ 
	(upper left) and $M_{3}^{+-}$ (upper right) with
	the Jenkin--Serrin graph boundary heights for the conjugates 
	of each eighth drawn below.}		\label{fig:mkpm_sketches}
\end{figure}
 In particular, we consider the periods along the curves
in the $(\beta_3, \beta_5)$ plane given by 
   \begin{eqnarray*}
      \tau_1=(0, \beta_5) \mbox{ for } \beta_5 \in [0,T], & 
        \hspace*{0.5in} & \tau_2=(\beta_3, T) \mbox{ for } \beta_3 \in [0,S], \\
      \tau_3=(S, \beta_5) \mbox{ for } \beta_5 \in [0,T], & & \tau_4=(\beta_3, 0)
         \mbox{ for } \beta_3 \in [0,S], 
   \end{eqnarray*}
for positive $S$ and $T$, and are able to 
show that, with the correct choices for $\beta_2$, $S$, and $T$, 
these curves surround a solution for the period problem.

 The conjugate contours for $M_1^{+-}$ associated
to points along the curves $\tau_1$ and $\tau_4$ degenerate to 
conjugate contours for either known surfaces or surfaces
which are known to have unsolvable period problems.
 When the degenerate contour is known to have a solvable
period problem, we assume nothing about the values of these
periods, and in general the remaining
unfixed parameter which we will not specify has been
shown to control this period.
 We seek to use only the general shape of the degenerate
contours and not the solvability of the period problems on 
the lower genus minimal surfaces.
 On each of the degenerate surfaces, the labels we use 
are inherited 
from the contour for $M_1^{+-}$, which may differ from those used
previously in the text.

\begin{proofspec}{Theorem \ref{thm:mkpm}}
Consider one-eighth of the fundamental piece, analogous to the 
depictions of $M_2^{+-}$ and $M_3^{+-}$ in 
Figure \ref{fig:mkpm_sketches}.  This one-eighth piece is bounded 
by seven planar geodesics $B_1$,...,$B_7$.  $B_1$ and $B_7$ are each of 
infinite length with a single endpoint, and $B_2$,...,$B_6$ are each 
finite length curve segments.  Let $\beta_j=$ Length($B_j$) for 
each $j=2,3,4,5,6$.  The singular points of the boundary 
are $V_j = B_j \cap B_{j+1}$ for $j = 1,...,6$.  (Unlike the 
cases when $k \ge 2$, there are no curves $C_j$, as in 
Figure \ref{fig:mkpm_sketches}.)  
We place the surface so that $g$ equals $1$ at the single end 
$E_1$ and equals $0$ at $V_1$, and we define the functions
   \begin{eqnarray}				\label{eq:mkpm_per}
      \pi_1(\beta_2,\beta_3,\beta_5) & = & \mbox{Re} \int_{V_3}^{V_4} \phi_2 
                  \; \;     = \; \mbox{Re} \int_{V_2}^{V_4} \phi_2, \nonumber \\
      \pi_2(\beta_2,\beta_3,\beta_5) & = & \mbox{Re} \int_{V_4}^{V_6} \phi_2 
                 \; \;     = \; \mbox{Re} \int_{V_5}^{V_6} \phi_2, \nonumber
   \end{eqnarray}
   where $\phi_2$ is the second component of the \Weierstra\ map
   given in equation (\ref{eq:WSD2}).  We will show that:
   \setcounter{i}{1}
   \begin{list}{\em \roman{i})}{\usecounter{i}}{\setlength{\leftmargin}{1in}}
      \item $\pi_1(\tau_j)$ and $\pi_2(\tau_j)$ change monotonically 
            on each $\tau_j$, for $j=1,2,3,4$.  In particular, 
for each fixed $\beta_2$ and $\beta_5$, $\pi_1(\beta_2,\beta_3,\beta_5)$ is 
a strictly decreasing function of $\beta_3$; 
for each fixed $\beta_2$ and $\beta_3$, $\pi_1(\beta_2,\beta_3,\beta_5)$ is 
a strictly increasing function of $\beta_5$; 
for each fixed $\beta_2$ and $\beta_5$, $\pi_2(\beta_2,\beta_3,\beta_5)$ is 
a strictly decreasing function of $\beta_3$; and 
for each fixed $\beta_2$ and $\beta_3$, $\pi_2(\beta_2,\beta_3,\beta_5)$ is 
a strictly decreasing function of $\beta_5$.  
      \item For all $\beta_2 > \epsilon$, $\pi_1(\beta_2,0,0) > 0$ and 
           $\pi_2(\beta_2,0,0) > 0$.  
      \item For any fixed $\beta_2 > \epsilon$, if $T$ is chosen sufficiently large, 
          then $\pi_1(\beta_2,0,T) > 0$ and $\pi_2(\beta_2,0,T) < 0$.  
      \item There exist choices for $\beta_2 > \epsilon$ and 
$S$ large so that $\pi_1(\beta_2,S,0) < 0$ and $\pi_2(\beta_2,S,0) = 0$.  
\end{list}
We consider the period map $\Pi(\beta_2,\beta_3,\beta_5) = 
(\pi_1(\beta_2,\beta_3,\beta_5), \pi_2(\beta_2,\beta_3,\beta_5))$.  We choose 
$\beta_2$, $S$, and $T$ so that 
$\pi_1(\beta_2,0,T) > 0$, $\pi_2(\beta_2,0,T) < 0$, 
$\pi_1(\beta_2,S,0) < 0$, and $\pi_2(\beta_2,S,0) = 0$.  
Since $\beta_2$ is then a fixed value, we may consider 
$\pi_1 = \pi_1(\beta_3,\beta_5)$ and 
$\pi_2 = \pi_2(\beta_3,\beta_5)$ as functions of only the two 
variables $\beta_3$ and $\beta_5$.  Hence 
$\Pi$ is a map from $\bfR^2$ to $\bfR^2$.  By the monotonic 
behavior of $\pi_1$ and $\pi_2$ on each $\tau_j$, 
it follows that the image of 
$\tau_1 \cup \tau_2 \cup \tau_3 \cup \tau_4$ under $\Pi$ is a homotopically 
nontrivial loop in $\bfR^2 \setminus \{ (0,0) \}$.  Thus 
a zero for the period map $\Pi$ lies 
   in the region bounded by the curves $\tau_j$.
Hence the period problem associated to $M_{1}^{+-}$ is solvable.  

We prove items (i),(ii),(iii),(iv) above by studying the conjugate 
surface of the original one-eighth portion bounded by planar geodesics 
$B_1$,...,$B_7$.  The conjugate surface is a graph $\cal B$ with respect to 
the $x_2$ direction over the 
rectangle $\{(x_1,0,x_3) \in \bfR^3 \, | \, 0 \leq x_1 \leq \beta_2, 
0 \leq x_3 \leq \epsilon \}$ in the $x_1x_3$-plane, and its boundary, 
the conjugate contour, consists of 
seven lines $B_1^*$,...,$B_7^*$ corresponding to the planar geodesics 
$B_1$,...,$B_7$ in the boundary of the original surface.  Since 
conjugation preserves lengths, we have 
$\beta_j = $Length($B_j$)$=$Length($B_j^*$).  Thus $B_1^*$ and 
$B_7^*$ are each infinite rays with a single endpoint, and 
$B_2^*$,...,$B_6^*$ are each finite line segments.  The singular points 
of the conjugate contour are $V_j^* = B_j^* \cap B_{j+1}^*$ for 
$j = 1,...,6$, corresponding to the points $V_j$ on the original surface.  
$B_1^*$ is the ray with endpoint $(\beta_2,-\beta_3,\epsilon)$ pointing 
in the direction of the positive $x_2$-axis.  
$B_2^*$ is the line segment with endpoints 
$(\beta_2,-\beta_3,\epsilon)$ and $(0,-\beta_3,\epsilon)$.  
$B_3^*$ is the line segment with endpoints 
$(0,-\beta_3,\epsilon)$ and $(0,0,\epsilon)$.  
$B_4^*$ is the line segment with endpoints 
$(0,0,\epsilon)$ and $(0,0,0)$.  
$B_5^*$ is the line segment with endpoints 
$(0,0,0)$ and $(0,\beta_5,0)$.  
$B_6^*$ is the line segment with endpoints 
$(0,\beta_5,0)$ and $(\beta_2,\beta_5,0)$.  
$B_7^*$ is the 
ray with endpoint $(\beta_2,\beta_5,0)$ pointing in the 
direction of the positive $x_2$-axis.  

We denote this conjugate graph by ${\cal B}(\beta_2,\beta_3,\beta_5)$, 
since it depends on the values of $\beta_2$, $\beta_3$ and $\beta_5$.  (It 
also depends on $\epsilon$, but $\epsilon$ will remain 
fixed, so we do not notate this dependence.)  

   {\bf Proof of (i):}
Choose nonnegative values $\beta_3$, $\tilde{\beta}_3$, and 
$\beta_5$, with $\beta_3 < \tilde{\beta}_3$, and choose any 
$\beta_2 > \epsilon$.  Then the interior 
of the graph 
${\cal B}(\beta_2,\beta_3,\beta_5)$ lies above the interior of 
${\cal B}(\beta_2,\tilde{\beta}_3,\beta_5)$ with respect to the 
$x_2$ direction, by Remark \ref{rmk:wayne}.  These two graphs have the 
line $B_4^*$ in common, and it follows that as one travels from 
$V_3^*$ to $V_4^*$ along $B_4^*$, the normal vector along 
$B_4^*$ of ${\cal B}(\beta_2,\beta_3,\beta_5)$ is turning ahead of 
the normal vector along $B_4^*$ of 
${\cal B}(\beta_2,\tilde{\beta}_3,\beta_5)$.  Furthermore, by the maximum 
principle these two normal vectors can never be equal in the interior 
of $B_4^*$.  This means that 
on the original surfaces the normal vector along $B_4$ for 
$\beta_2,\beta_3,\beta_5$ is turning strictly ahead of the normal vector along
$B_4$ for $\beta_2,\tilde{\beta}_3,\beta_5$, with respect to arc length.
Since Length($B_4$)$=$Length($B_4^*$)$=\beta_4=\epsilon$ is independent of 
$\beta_3$, it follows that $\pi_1(\beta_2,\beta_3,\beta_5) 
> \pi_1(\beta_2,\tilde{\beta_3},\beta_5)$.  

We have just shown that for each fixed $\beta_2$ and $\beta_5$, $\pi_1$ is a 
strictly decreasing function of $\beta_3$.  Similar arguments show the 
other parts of {\bf (i)}.  

   {\bf Proof of (ii):}
If $\beta_3 = \beta_5 = 0$, then $V_2$ coincides with $V_3$ and 
$V_4$ coincides with $V_5$.  The conjugate graph of this surface 
is unique, by Theorem \ref{thm:Jen-Ser}, hence the surface is 
unique.  Therefore it is $M_{1}$.  The embeddedness of 
$M_{1}$ implies that $\pi_2(\beta_2,0,0) > 0$.  

The surface $M_1$ contains a vertical line, and this line divides both 
$M_1$ and $B_4$ into two congruent pieces.  Let $\hat{B}_4$ be the 
half of $B_4$ that connects the midpoint of $B_4$ to $V_3 = V_2$.  Let 
$\hat{M}_1$ be the congruent piece of $M_1$ bounded by $B_1$, $B_2$, 
$\hat{B}_4$, and the vertical line.  Since $\hat{M}_1$ has a single 
Scherk-type end 
whose normal is parallel to the $x_1$ axis, the maximum principle implies that 
the $x_2$ coordinate function on $M_1$ cannot be maximized in the interior of 
$M_1$.  Furthermore, as $B_2$ is a planar geodesic in a plane parallel to the 
$x_2x_3$-plane, the boundary maximum principle implies that $x_2$ cannot be 
maximized on $B_2$.  Similarly, $x_2$ cannot be maximized on the interior of 
$\hat{B}_4$.  Therefore the value of the $x_2$ coordinate at $V_2=V_3$ is 
strictly less than the value of the $x_2$ coordinate at the midpoint of 
$B_4$.  So $\pi_1(\beta_2,0,0) > 0$.  

   {\bf Proof of (iii):} 
Fix $\beta_2 > \epsilon$, and 
choose $\beta_3=0$ and $\beta_5= T >> 1$.  Then 
$\lim_{T \to \infty} {\cal B}(\beta_2,0,T)$ is a graph bounded by 
$B_1^*$, $B_2^*$, $B_4^*$, and an infinite ray with endpoint at 
$V_4^*$ pointing in the direction of the positive $x_2$-axis.  
This graph has a single Scherk--type end of width 
$\sqrt{\beta_2^2 + \epsilon^2}$.  (The fact that this limiting 
behavior occurs follows from the arguments in \cite{jes1}.  In this 
proof we will consider various limit surfaces, and in all cases the 
existence of the limit graph follows from \cite{jes1}.)  

The original surface corresponding to 
$\lim_{T \to \infty} {\cal B}(\beta_2,0,T)$ via conjugation 
is bounded by the planar geodesics $B_1$, $B_2$, $B_4$, and an 
infinite version of $B_5$.  It has a single non-vertical Scherk--type 
end of width $\sqrt{\beta_2^2 + \epsilon^2}$.  
On $\lim_{T \to \infty} {\cal B}(\beta_2,0,T)$, the maximum principle 
implies that its normal vector $\vec{N}$ along $B^*_4$ lies 
within a $90^o$ geodesic arc of the unit sphere (so this is also true 
along $B_4$), and thus the $x_2$ coordinate at $V_4$ is greater 
than the $x_2$ coordinate at $V_2 = V_3$ on the 
original surface, so $\lim_{T \to \infty} \pi_1(\beta_2,0,T) > 0$.  Hence 
$\pi_1(\beta_2,0,T) > 0$ for $T$ sufficiently large.  

Now we consider the limiting conjugate surface 
$\lim_{T \to \infty} ({\cal B}(\beta_2,0,T) - (0,T,0))$, which is 
a graph bounded by 
$B_6^*$, $B_7^*$, an infinite version of $B_5^*$ equal to the negative 
$x_2$ axis, and a complete line through 
$(\beta_2,0,\epsilon)$ parallel to the the $x_2$-axis.  
This conjugate surface has two ends of Scherk--type.  One end has 
width $\epsilon$ and the other has width 
$\sqrt{\beta_2^2 + \epsilon^2}$.  The original surface that corresponds 
to it via conjugation is bounded by $B_6$, $B_7$, an infinite version 
of $B_5$, and a 
complete infinite version of $B_1$.  It has two ends, again of 
width $\epsilon$ and $\sqrt{\beta_2^2 + \epsilon^2} > \epsilon$.  Because 
of the relative widths of the ends on this original surface, we see 
that the $x_2$ coordinate at $V_5$ is greater 
than the $x_2$ coordinate at $V_6$, so 
$\lim_{T \to \infty} \pi_2(\beta_2,0,T) < 0$.  Hence 
$\pi_2(\beta_2,0,T) < 0$ for $T$ sufficiently large.  (See Figure 
\ref{addedfig}.)  

\begin{figure}
\begin{center}
\unitlength=1.0pt
\begin{picture}(-200.00,200.00)(0.00,300.00)
\put(-16.00,270.00){\makebox(0,0)[cc]{$B_7$}}
\put(-203.00,80.00){\makebox(0,0)[cc]{$x_1$}}
\put(-224.00,103.00){\makebox(0,0)[cc]{$x_3$}}
\put(-214.00,88.00){\makebox(0,0)[cc]{$x_2$}}
\put(23.00,189.00){\makebox(0,0)[cc]{$B_6$}}
\put(-60.00,166.00){\makebox(0,0)[cc]{"$B_5$"}}
\put(-93.00,212.00){\makebox(0,0)[cc]{"$B_1$"}}
\put(-133.00,108.00){\makebox(0,0)[cc]{$\sqrt{\beta_2^2+\epsilon^2}$}}
\put(-27.00,307.00){\makebox(0,0)[cc]{$\epsilon$}}
\put(-230,80){\vector(1,0){20}}
\put(-230,80){\vector(0,1){20}}
\put(-230,80){\vector(2,1){10}}
\put(-200,100){\line(1,0){200}}
\put(-200,300){\line(1,0){200}}
\put(-200,100){\line(0,1){200}}
\put(0,100){\line(0,1){200}}
\put(0,100){\line(2,1){40}}
\put(-200,300){\line(2,1){40}}
\put(0,300){\line(2,1){40}}
\put(40,120){\line(0,1){200}}
\put(-160,320){\line(1,0){200}}
\bezier20(-200,100)(-180,110)(-160,120)
\bezier10(-48,300)(-32,308)(-16,316)
\bezier10(-180,100)(-160,110)(-140,120)
\bezier600(-180,100)(-48,232)(-48,300)
\bezier80(-160,120)(-160,220)(-160,320)
\bezier80(-160,120)(-60,120)(40,120)
\bezier80(40,200)(10,185)(10,210)
\bezier80(32,240)(10,229)(10,210)
\bezier300(32,240)(-16,240)(-16,316)
\bezier450(-140,120)(-60,200)(40,200)
\end{picture}
\vspace{4in}
\end{center}
    \caption{The original limit surface described at the end of the 
proof of (iii).}%

\label{addedfig}
\end{figure}
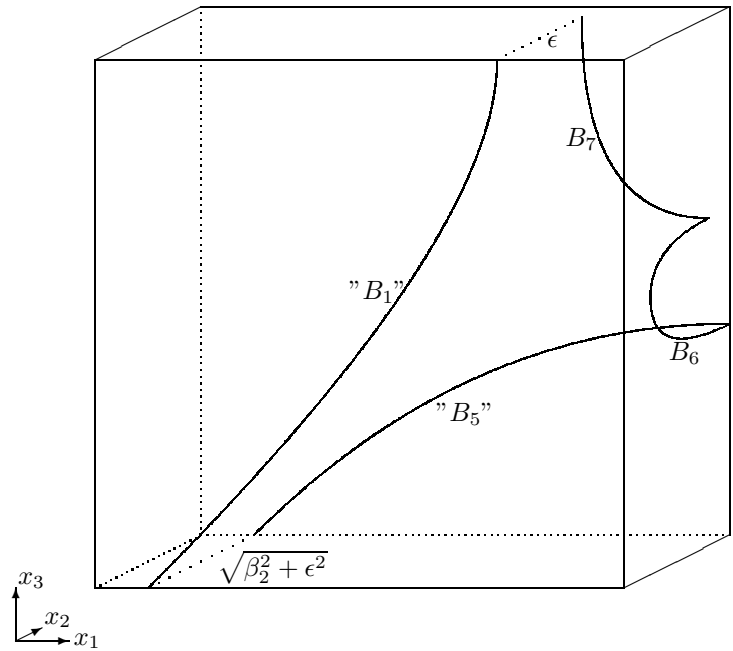

   {\bf Proof of (iv):}
Choose $\beta_2 > \epsilon$, 
$\beta_5=0$, and $\beta_3= S >> 1$.  We consider the limiting 
conjugate surface 
$\lim_{S \to \infty} {\cal B}(\beta_2,S,0)$, which is 
a graph bounded by 
$B_4^*$, $B_6^*$, $B_7^*$, an infinite ray with endpoint 
at $V_3^*$ pointing in the direction of the negative $x_2$-axis, 
and a complete line through 
$(\beta_2,0,\epsilon)$ parallel to the the $x_2$-axis.  
This conjugate surface has two ends of Scherk--type.  One end has 
width $\epsilon$ and the other has width 
$\beta_2$.  The original surface that corresponds 
to it via conjugation is bounded by $B_4$, $B_6$, $B_7$, an 
infinite ray version of $B_3$, and a 
complete infinite version of $B_1$.  It has two ends, again one of 
width $\epsilon$ and the other of width $\beta_2$.  

We now consider what happens to the original surface corresponding 
to $\lim_{S \to \infty} {\cal B}(\beta_2,S,0)$ as $\beta_2 
\searrow \epsilon$ and as $\beta_2 \nearrow \infty$.  

The conjugate surface $\lim_{\beta_2 \to \epsilon} 
(\lim_{S \to \infty} {\cal B}(\beta_2,S,0))$ is a graph with 
respect to the $x_2$ direction over the 
square $\{(x_1,0,x_3) \in \bfR^3 \, | \, 0 \leq x_1,  
x_3 \leq \epsilon \}$.  It is bounded by 
the infinite ray $B_7^*$ with endpoint $(\epsilon,0,0)$ pointing in the 
direction of the positive $x_2$-axis, the line segment $B_6^*$ from 
$(\epsilon,0,0)$ to $(0,0,0)$, the line segment $B_4^*$ from 
$(0,0,0)$ to $(0,0,\epsilon)$, and the infinite ray with 
endpoint $(0,0,\epsilon)$ pointing in the direction of the negative 
$x_2$-axis.  The corresponding original surface in bounded by the 
planar geodesics $B_4$, $B_6$, $B_7$, and a complete infinite version 
of $B_1$.  This original surface has two ends of Scherk--type, both 
of width $\epsilon$.  

Note that the graph $\lim_{\beta_2 \to \epsilon} 
(\lim_{S \to \infty} {\cal B}(\beta_2,S,0))$ contains the line 
segment from $(0,0,0)$ to $(\epsilon,0,\epsilon)$ and is symmetric 
with respect to rotation about this line, by uniqueness in Theorem 
\ref{thm:Jen-Ser} and by Theorem \ref{thm:schwarz}.  
The maximum principle then implies that the normal vector $\vec{N}$ along 
each of $B_4^*$ and $B_6^*$ is contained in a $90^o$ geodesic 
arc of the unit sphere, and thus the $x_2$ coordinate at $V_3$ is greater 
than the $x_2$ coordinate at $V_4 = V_5$ in the corresponding 
original surface, and the $x_2$ coordinate at $V_3$ equals 
the $x_2$ coordinate at $V_6$.  
Therefore $\lim_{\beta_2 \to \epsilon} 
(\lim_{S \to \infty} \pi_2(\beta_2,S,0)) = - 
\lim_{\beta_2 \to \epsilon} 
(\lim_{S \to \infty} \pi_1(\beta_2,S,0)) 
 > 0$.  Hence for $\beta_2$ sufficiently close to $\epsilon$ and 
$S$ sufficiently large, we have $\pi_2(\beta_2,S,0) > 0$.  

The limiting conjugate surface $\lim_{\beta_2 \to \infty} 
(\lim_{S \to \infty} {\cal B}(\beta_2,S,0))$ is a portion of a helicoid 
(this follows from \cite{jes1}) 
bounded by the positive $x_1$-axis, the line 
segment $B_4^*$ from $(0,0,0)$ to $(0,0,\epsilon)$, and 
an infinite ray with endpoint $(0,0,\epsilon)$ pointing in the direction 
of the negative $x_2$-axis.  On the corresponding original surface, one 
eighth of a catenoid, we 
then have that $B_4$ is a quarter circle of radius $2\epsilon/\pi$.  
Thus $\lim_{\beta_2 \to \infty} 
(\lim_{S \to \infty} \pi_1(\beta_2,S,0)) = -2\epsilon/\pi$.  Since the 
original surface corresponding to 
$\lim_{S \to \infty} {\cal B}(\beta_2,S,0)$ has two Scherk--type ends of 
width $\epsilon$ and $\beta_2$, it follows that $\lim_{\beta_2 \to \infty} 
(\lim_{S \to \infty} (\pi_1(\beta_2,S,0) + \pi_2(\beta_2,S,0))) = 
\lim_{\beta_2 \to \infty} (\epsilon - \beta_2) = - 
\infty$.  Thus for $\beta_2$ and 
$S$ sufficiently large, we have $\pi_2(\beta_2,S,0) < 0$.  

Therefore for some large $S$ and some value of $\beta_2 > 
\epsilon$, we have $\pi_2(\beta_2,S,0) = 0$.  If, for this 
$S$ and $\beta_2$, we 
have $\pi_1(\beta_2,S,0) \geq 0$, then the original surface corresponding 
to this $\beta_2$, $\beta_3=S$, and $\beta_5=0$ would contain 
some point in $B_4 \cup B_6$ where $x_2$ has a local maximum and where 
the tangent plane is parallel to the $x_1x_3$-plane.  This 
contradicts the maximum principle.  So, for this $S$ and 
$\beta_2$, we have $\pi_1(\beta_2,S,0) < 0$.  This 
shows {\bf (iv)}.  

    This completes the proof of the solvability of the period
   problem associated to $M_{1}^{+-}$.  Note that 
$\Pi(\tau_1 \cup \tau_2 \cup \tau_3 \cup \tau_4)$ changes continuously 
under continuous changes of $\beta_2$, so for all 
$\beta_2$ sufficiently close to the $\beta_2$ chosen above, 
$\Pi(\tau_1 \cup \tau_2 \cup \tau_3 \cup \tau_4)$ is still a homotopically 
nontrivial loop in $\bfR^2 \setminus \{(0,0)\}$, and so the 
period problem remains solvable.  
  Hence $\beta_2$ in a small open interval serves as a 
   deformation parameter for the surface, thereby yielding a 
   one--parameter family of these surfaces.
    Since each eighth of the surface is embedded, and the maximum
   principle tells us this embedded surface lies in the bounding 
   box determined by the planar curves $B_j$, each surface
   in the family is embedded.
    This completes the proof.
\end{proofspec}

 The proof of Theorem \ref{thm:mkpm}
cannot be directly adapted to prove existence of $M_{k}^{+-}$ for 
$k \ge 2$.  However, 
numerical evidence suggests that the $M_{k}^{+-}$ exist for $k \ge 2$ as 
well, so we make this conjecture (see Figure \ref{fig:intro_images2}).

\begin{conjecture}					\label{conj:mkpm}
    There exists a one-parameter family of genus three,
   embedded minimal surfaces $M_{k}^{+-}$ with $4k$ Scherk-type
   ends, for all $k \ge 2$.
\end{conjecture}

In fact, numerical evidence also suggests that there exists a wide 
variety of minimal surfaces with Scherk-type ends and more handles of both 
$+$ and $-$ type.  (See Figure \ref{fig:2k+_imageimage}.)  

\section{Appendix: a proof of Lemma 4.1}

The first part of Lemma 4.1 is intended only to be 
an intuitive aid, saying that "each collection of surfaces results from 
adding ends and handles to $M_1$".  However, a rigorous proof is required 
for the statement 
that "the period problems arising from the additional ends
are all solved by requiring $\epsilon = \mbox{length}(A_i) = \mbox{length}(B_j)$".  

For each surface, we always begin by choosing one-eighth of the original 
fundamental piece of the surface.  This one-eighth piece is bounded by 
planar geodesics, and its conjugate surface is bounded by portions of 
lines.  Before we consider any period problems, we must first establish
existence of this conjugate surface, which then implies the 
existence of the original one-eighth piece 
(without solving for period problems yet).  

The conjugate pieces exist because they are Jenkins-Serrin graphs.  In all the 
cases we consider, 
they are Jenkins-Serrin graphs over a rectangle, and the boundary 
data is a finite constant over each of three sides of the boundary of the
rectangle.  On the fourth side, the boundary data alternates between $+\infty$ and 
$-\infty$ along adjacent 
intervals.  The jump discontinuities occur only at the corners 
of the rectangle and at points along the fourth side where the boundary data 
changes from $+\infty$ to $-\infty$.  

Recall Theorem \ref{thm:Jen-Ser}.  In our case, by 
applying a rigid motion and a homothety of $\bfR^3$, we may assume without
loss of generality that $D = \{(x_1,x_2) \in \bfR^2 \; | \; 0 \leq x_1 \leq \delta, 
0 \leq x_2 \leq 1 \}$ for some positive $\delta$, that there are three $C_k$'s which 
we define as $C_1 = D \cap \{ (x_1,0) \in \bfR^2 \}$, 
$C_2 = D \cap \{ (\delta,x_2) \in \bfR^2 \}$, and $C_3 = D \cap \{ (x_1,1) 
\in \bfR^2 \}$, and that there are 
$\ell$ number of $A_i$'s and $B_j$'s, all of length $1 \over \ell$ alternating 
along $D \cap \{ (0,x_2) \in \bfR^2 \}$.  
Notice here that we have already incorporated the condition of Lemma 4.1, that is, 
that 
   \begin{displaymath}
	\epsilon = \mbox{length}(A_i) = \mbox{length}(B_j) = {1 \over \ell} \; . 
   \end{displaymath}

Existence and uniqueness 
of a solution $u$ to the minimal surface equation with the given boundary 
data now follows from Theorem \ref{thm:Jen-Ser}.  (The conditions on the 
polygons $\cal P$ are trivially satisfied, since no such $\cal P$ exists 
for the boundary conditions we are using.)  
Furthermore, the results in \cite{jes1} imply that 
$u$ is finite at every point in the interior of $D$.  

Let $M$ denote the smallest closed minimal surface that contains the graph $u$.  
(Here we use the word "closed" in the sense that $M$ contains all of its accumulation 
points.)  Hence the interior of $M$ is the interior of the 
graph $u$, and $M$ contains its boundary 
$\partial M$, and the image of the vertical projection of $M$ to the $x_1x_2$-plane is 
$D \setminus \{(0,x_2) \in \bfR^2 \, | \, x_2 
\neq {k \over \ell} \mbox{ some } k \in \bfZ \}$.  

We now prove Lemma 4.1 in a series of eight claims.  

\begin{quote}
{\bf Claim 1:} $M$ has finite total absolute curvature.
\end{quote}

\begin{proof}
Following the proof in \cite{jes1}, p 334, $u$ is the limit of a subsequence 
of minimal graphs $\{u_n\}_{n=1}^\infty$ over 
Int($D$).  The minimal graphs $u_n$ are determined by 
replacing the boundary condition $+\infty$ by $+n$ on each $A_i$, 
replacing the boundary condition $-\infty$ by $-n$ on each $B_j$, and leaving the 
boundary data on $C_1 \cup C_2 \cup C_3$ unchanged.  

First we show that the total absolute 
curvature of the graph of $u_n$ is bounded above by a finite bound 
independent of $n$, which follows 
easily from the Gauss-Bonnet formula.  The boundary of $u_n$ is polygonal with at 
most $2\ell + 6$ boundary line segments, and at each intersection of adjacent 
boundary line segments the angle of intersection is $\pi \over 2$.  Hence 
the total geodesic curvature of the boundary curve for the graph $u_n$ is 
at most ${\pi \over 2} (2\ell +6)$.  The Gauss-Bonnet formula then implies 
\begin{equation}\label{pf1eqn1} 
\int_{Graph(u_n)} |K| dA \leq \pi (\ell + 1) \; \; \; \; \forall n \; , 
\end{equation}
where $dA$ is the area form on Graph($u_n$) induced by $\bfR^3$, and $K$ is the 
intrinsic Gaussian curvature of Graph($u_n$).  

Now we claim that for 
any compact convex domain $D^\prime \subset $Int($D$), there exists a subsequence 
$\{n_j\}_{j=1}^\infty$ such that the total absolute curvature of the 
graphs of $u_{n_j}$ restricted to the domain $D^\prime$ converges to the total 
absolute curvature of the graph of 
$u$ restricted to $D^\prime$.  That is, we claim that 
\begin{equation} \label{pf1eqn2}
\int_{Graph(u_{n_j}|_{D^\prime})} |K|dA \to \int_{Graph(u|_{D^\prime})}|K|dA 
\end{equation} 
as $n_j \to \infty$.  This follows from the fact that, as shown in \cite{jes1},  
$u_{n}|_{D^\prime}$ converges uniformly to $u|_{D^\prime}$ as $n \to \infty$.  
The convergence (\ref{pf1eqn2}) is essentially known, and arguments showing it 
exist in several places.  For example, a proof is contained in the arguments 
proving Theorem 2 in \cite{MY}.  The arguments in \cite{MY} are intended for 
more general ambient spaces, and when the ambient space is $\bfR^3$ the 
arguments in \cite{MY} can be considerably simplified.  A simpler argument for 
the $\bfR^3$ case can be found in section III.2 of \cite{C}.  

For completeness, in this paragraph we outline an argument showing 
(\ref{pf1eqn2}).  We know that the 
$u_n$ converge uniformly to $u$ over $D^\prime$, by \cite{jes1}.  These 
graphs $u_n|_{D^\prime}$ (resp. $u|_{D^\prime}$) are graphs over 
convex domains in the $x_1x_2$-plane and hence are the unique compact 
minimal surfaces with respect to their boundaries.  Hence 
they coincide as surfaces in $\bfR^3$ with the Douglas-Rado solutions 
$f_n: B^2:=\{(u,v) \in \bfR^2 \, | \, u^2+v^2 \leq 1 
\} \to \bfR^3$ (resp. $f: B^2 \to \bfR^3$) for their boundaries.  
That is, the surfaces $f_n(B^2)$ and $\{(x_1,x_2,u_n(x_1,x_2)) \in \bfR^3 \, | \, 
(x_1,x_2) \in D^\prime \}$ (resp. $f(B^2)$ and $\{(x_1,x_2,u(x_1,x_2)) \in \bfR^3 \, | \, 
(x_1,x_2) \in D^\prime \}$) coincide.  The parametrizations $f_n$ and $f$ have the 
advantage that they are conformal, hence the coordinate functions 
$f_n^i,f^i,i=1,2,3$ are harmonic on $B^2$.  Using 
arguments similar to those we use later to prove Claim 6 of this note, we can see
that in fact 
\[ \frac{\partial u_n}{\partial x_1} \to \frac{\partial u}{\partial x_1}, 
   \frac{\partial u_n}{\partial x_2} \to \frac{\partial u}{\partial x_2} \] 
converge uniformly over $D^\prime$ as well.  (This is equivalent to showing that 
the normal vectors of the graphs converge uniformly over $D^\prime$.)  Once we 
know that first derivatives of $u_n$ also converge 
uniformly, the arguments in the proof of Lemma 3.2 and the remark following it in 
\cite{C} can be applied: using the three-point condition as in \cite{C}, we can 
find a subsequence $f_{n_j}$ of the $f_n$ which converge uniformly to 
$f$ on $\partial B^2$.  Since the functions $f_{n_j}^i,f^i$ are harmonic, and 
hence the functions $|f_{n_j}^i - f^i|$ always attain their 
maximums on $\partial B^2$, we conclude that $f_{n_j} \to f$ uniformly on all of 
$B^2$.  Uniform convergence for harmonic functions implies that the convergence is 
smooth (this is a basic property of harmonic functions, see, for example, Theorem 
2.10 of \cite{GT}).  We conclude that the Douglas-Rado solutions $f_{n_j}$ converge 
smoothly to $f$.  Hence the total absolute curvature of the graphs 
$u_{n_j}|_{D^\prime}$ converges to the total absolute 
curvature of the graph $u|_{D^\prime}$.  This shows the convergence (\ref{pf1eqn2}).  

If the total absolute curvature of $M$ is strictly greater than $\pi (\ell+1)$, then 
there exists some compact convex domain $D^\prime \subset $Int($D$) such that 
the graph of $u|_{D^\prime}$ has total absolute curvature strictly greater 
than $\pi (\ell+1)$.  However, then the convergence (\ref{pf1eqn2}) contradicts 
equation (\ref{pf1eqn1}).  Therefore the total absolute 
curvature of $M$ is at most $\pi (\ell + 1)$, and Claim 1 is shown.  
\end{proof}

\begin{quote} {\bf Claim 2:} 
The are only a finite number of points of $M$ at which the tangent plane 
is horizontal.  
\end{quote}

\begin{proof}
The proof below is simply a modification of an argument in the proof of 
Theorem 3.1 of \cite{MW}.  

Consider the Gauss map $G:M \to S^2 = \{(x_1,x_2,x_3) \in \bfR^3 \, | \, x_1^2+x_2^2+
x_3^2=1 \}$.  $M$ is the closure of the graph $u$, so 
$M$ is orientable, and so $G$ is well-defined.  We can define $G$ so that $G(M) 
\subset S^2 \cap \{(x_1,x_2,x_3) \in \bfR^3 \, | \, x_3 \geq 0 \}$.  
With respect to conformal coordinates on $M$, $G$ is a holomorphic map from $M$ to the 
upper hemisphere of $S^2$, hence $G$ is a branched covering with boundary 
into the upper hemisphere.  Furthermore, since $\partial M$ consists of 
portions of lines parallel to the coordinate axes in $\bfR^3$, 
\[ G(\partial M) \subset \{(x_1,x_2,x_3) \in S^2 \, | \, x_1=0 \mbox{ or } 
x_2=0 \mbox{ or } x_3=0 \} \cap \{(x_1,x_2,x_3) \in \bfR^3 \, | \, x_3 \geq 0 \} \; . \]  
Therefore the covering degree of $G$ is a constant on each of the four sets 
\[ \{(x_1,x_2,x_3) \in S^2 \, | \, x_1>0,x_2>0,x_3>0 \} \; , \; \; \; \; \;  
\{(x_1,x_2,x_3) \in S^2 \, | \, x_1<0,x_2>0,x_3>0 \} \; , \] 
\[ \{(x_1,x_2,x_3) \in S^2 \, | \, x_1>0,x_2<0,x_3>0 \} \; , \; \; \; \; \; 
\{(x_1,x_2,x_3) \in S^2 \, | \, x_1<0,x_2<0,x_3>0 \} \; . \] 
By Claim 1, these four constant covering degrees are 
all finite.  If the inverse image $G^{-1}(\vec{e}_3 = (0,0,1) )$ were 
to contain infinitely many points of $M$, then at least one of these 
four constant covering degrees would not be finite.  Hence 
$G^{-1}(\vec{e}_3 = (0,0,1) )$ is a finite set, showing Claim 2.  
\end{proof}

Let $P_s = \{(x_1,x_2,s) \in \bfR^3 \}$ be the horizontal plane 
in $\bfR^3$ of height $s$.  An immediate corollary to Claim 2 is the following 
Claim 3.  In Claim 3, by "nonsingular curves of $\bfR^3$", we mean curves of 
$\bfR^3$ which are 1-dimensional submanifolds with boundary.  

\begin{quote} {\bf Claim 3:} 
There exists a constant $L>0$ such that, for all $L^\prime > L$, $P_s \cap 
M, s \in [L,L^\prime]$ (resp. $s \in [-L^\prime,-L]$) is a smooth 
deformation (with respect to $s$) from $P_{L} \cap M$ to $P_{L^\prime} 
\cap M$ (resp. from $P_{-L} \cap M$ to $P_{-L^\prime} \cap 
M$) through an embeddeded collection of nonsingular curves of $\bfR^3$.  
\end{quote}

\begin{proof}
A singularity in this deformation can only occur at a point of $M$ where the tangent 
plane is horizontal.  By Claim 2, we can choose $L$ large enough that no such 
horizontal points exist in $\{(x_1,x_2,x_3) \in M 
\, | \, x_3 \geq L \}$ nor in $\{(x_1,x_2,x_3) \in M \, | \, x_3 \leq -L \}$.  This 
proves Claim 3.  
\end{proof}

Thus, by Claim 3, for any $L^\prime > L$, $M \cap \{(x_1,x_2,x_3) \in \bfR^3 \, | \, 
x_3 \in [L,L^\prime ] \}$ consists of a finite number of 
components, and each component is an embedded disk bounded 
by two vertical lines segments, and one curve in $P_{L^\prime}$, and one 
curve in $P_L$.  We choose any component $M_{comp}$ of $M \cap \{(x_1,x_2,x_3) \in \bfR^3 \, 
| \, x_3 \in [L,L^\prime ] \}$ and extend it by rotations of $\pi$ radians about 
vertical boundary lines (this can be done, and the extended surfaces are 
smooth, by the Schwarz reflection principle, Theorem 3.1).  
Extending $M_{comp}$ (and its extended surfaces) by these rotations a finite 
number of times results in a larger compact 
surface which still has only two vertical boundary line segments, and one 
boundary curve in $P_{L^\prime}$, and one boundary curve in $P_L$.  We make 
these rotational 
extensions enough times so that the distance in $\bfR^3$ from any point 
in $M_{comp}$ to the two boundary vertical line segments of the extended surface 
is greater than ${1 \over 4}(L^\prime - L)$.  We call this extended surface 
$M_{comp}^{ext}$ -- it is an immersed compact disk in $\bfR^3$, and is 
not necessarily embedded.  (We will later see that $M_{comp}^{ext}$ is indeed 
embedded for $L$ large enough.)  

\begin{quote} {\bf Claim 4:} $M_{comp}^{ext}$ is strongly stable.  
\end{quote}

\begin{proof}
The image $G(M_{comp})$ is contained in the upper hemisphere of $S^2$ and 
does not contain the north pole $\vec{e}_3$.  Since 
$M_{comp}^{ext}$ is comprised of a finite number of pieces 
congruent to $M_{comp}$ which are all images of vertical rotations 
of $M_{comp}$, it follows that $G(M_{comp}^{ext})$ is also 
contained in the upper hemisphere and does not contain 
$\vec{e}_3$.  In particular, the area of $G(M_{comp}^{ext})$ in $S^2$ is 
strictly less than $2\pi$.  

Theorem 1.2 of \cite{BdC} tells us that if the area of $G(M_{comp}^{ext})$ is 
less than $2\pi$, then $M_{comp}^{ext}$ is stable.  The map $G$ is not required 
to be an injection in order for this theorem to hold, and the minimal surface 
need only be an immersion -- it does not need to be an embedding.  
Furthermore, in \cite{BdC} the word {\em stable} is used 
in the strong sense; that is, a minimal surface is stable 
if the second derivative of area for any smooth nontrivial 
boundary-preserving variation is {\em strictly} positive.  This shows Claim 4.  
\end{proof}

For an oriented minimal surface ${\cal M} \subset \bfR^3$, 
let dist$_{\cal M}(A,B)$ be the intrinsic distance in 
$\cal M$ between two sets $A,B \subset {\cal M}$.  
For each point $q \in {\cal M}$, let $K_q$ be the Gaussian curvature of 
$\cal M$ at $q$, and let $\vec{N}_q$ be the oriented unit normal vector 
of $\cal M$ at $q$.  Let $\vec{e}_1 = (1,0,0)$, and let $\langle \cdot , 
\cdot \rangle$ be the standard inner product on $\bfR^3$.  Let 
dist$_{\bfR^3}(A,B)$ be the distance in 
$\bfR^3$ between two sets $A,B \subset {\bfR^3}$.  

By Corollary 4 of \cite{S} there exists a universal constant $c$ such that 
\[ |K_q| < {c \over (\mbox{dist}_{\cal M}(q,\partial {\cal M}))^2} \; , \] where 
$\cal M$ is any compact stable minimal surface in $\bfR^3$, and $q$ is any 
point in $\cal M$.  
This result (just like Theorem 1.2 of \cite{BdC}) does not require 
the surface $\cal M$ to be embedded -- 
only immersed.  The constant $c$ is universal in the sense that it is independent 
of the choice of $\cal M$.  (See Theorem 16.20 of \cite{GT} and Theorem 11.1 of 
\cite{O} for related results.)  

\begin{quote}
{\bf Claim 5:} On the surface $\hat{M} := M_{comp} \cap 
\{ (x_1,x_2,x_3) \in \bfR^3 \, | \, x_3 \in 
[{1 \over 4}L^\prime + {3 \over 4}L,{3 \over 4}L^\prime + {1 \over 4}L] \} $, 
the Gaussian curvature $K$ is uniformly bounded by 
\[ |K| < \frac{16c}{(L^\prime - L)^2} \; . \]
\end{quote}

\begin{proof}
$M_{comp}^{ext}$ is a compact minimal surface in $\bfR^3$, which is strongly 
stable by Claim 4.  For all $q \in \hat{M}$, 
dist$_{M_{comp}^{ext}}(q,\partial M_{comp}^{ext}) \geq {L^\prime - L \over 4}$.  
Now we apply Corollary 4 of \cite{S} and Claim 5 is proven.  
\end{proof}

Assume that $L^\prime$ is chosen large enough that 
${8\sqrt{c} \delta \over L^\prime - L } < 1$.  

\begin{quote}
{\bf Claim 6:} At every point of $M_{comp} \cap 
\{(x_1,x_2,x_3) \in \bfR^3 \, | \, x_3 \in [{3 \over 8}L^\prime + {5 \over 8}L,{5 
\over 8}L^\prime + {3 \over 8}L] \}$, we have 
\[ | \langle \vec{N},\vec{e}_1 \rangle | \geq \sqrt{1-{8 \sqrt{c} \delta 
\over L^\prime-L}} \; . \] 
\end{quote}

\begin{proof}
Suppose some point $p \in M_{comp} \cap \{(x_1,x_2,x_3) \in \bfR^3 \, | \, x_3 \in 
[{3 \over 8}L^\prime + {5 \over 8}L,{5 \over 8}L^\prime + {3 \over 8}L] \}$ has 
normal $\vec{N}_p$ so that 
$|\langle \vec{N}_p , \vec{e}_1 \rangle| < \sqrt{1-{8 \sqrt{c} \delta \over 
L^\prime - L}}$.  Then 
there is a tangent vector $\vec{T}$ at $p$ such that 
$\langle \vec{T} , \vec{e}_1 \rangle > \sqrt{{8 \sqrt{c} \delta \over L^\prime - 
L}}$.  Assume $L^\prime$ and $c$ are chosen large enough that 
$L^\prime -L > 1$ and $ c > 1024 \cdot \delta^2$.  
Consider a unit-speed geodesic $\gamma(t) \subset \hat{M}, 
t \in [0,(L^\prime-L)/8]$ so that $\gamma(0) = p$ and $\gamma^\prime (0) = 
\vec{T}$, where $\prime = \frac{\partial}{\partial t}$.  We define 
\[ t_0 := \sqrt{{\delta (L^\prime - L) \over 2 \sqrt{c}}} \; . \] 
Since $L^\prime > L+1$ and $c > 1024 \cdot \delta^2$, we have that 
$t_0 < (L^\prime-L)/8$ and hence $\gamma (t_0) \in \hat{M}$.  Let $k_g(t)$ be the 
geodesic curvature of $\gamma(t)$.  Since 
$|K_q| < \frac{16c}{(L^\prime - L)^2}$ for all $q \in \hat{M}$ by Claim 5, and 
since $\hat{M}$ is minimal, $|k_g(t)| < 
\frac{4\sqrt{c}}{(L^\prime - L)}$ for all $t \in [0,(L^\prime-L)/8]$.  Thus 
$|\gamma^{\prime \prime}(t)| < \frac{4\sqrt{c}}{(L^\prime - L)}$.  Writing 
$\gamma(t) = (\gamma_1(t),\gamma_2(t), 
\gamma_3(t))$ in terms of coordinates in $\bfR^3$, we have 
$|\gamma_1^{\prime \prime}(t)| < \frac{4\sqrt{c}}{(L^\prime - L)}$.  Then for 
$t \in [0,(L^\prime-L)/8]$, 
\[ |\gamma_1^\prime(t) - \gamma_1^\prime(0)| = 
\left| \int_0^t \gamma_1^{\prime \prime}(s) ds \right| \leq 
\int_0^t |\gamma_1^{\prime \prime}(s) | ds < \frac{4\sqrt{c}}{L^\prime - L} 
\cdot t \; \; , \]
and thus $\gamma_1^\prime(t) > \gamma_1^\prime(0) - 
\frac{4\sqrt{c}}{L^\prime - L} \cdot t$.  Therefore 
\[ \gamma_1(t_0) \geq \gamma_1(t_0) - \gamma_1 (0) = \int_0^{t_0} 
\gamma_1^\prime(t) dt > \int_0^{t_0} \left( \gamma_1^\prime(0) - 
{4 \sqrt{c} \over L^\prime -L} t \right) dt = \]
\[ = \gamma_1^\prime (0) t_0 - {2 \sqrt{c} \over L^\prime -L} t_0^2
> \sqrt{{8 \sqrt{c} \delta \over L^\prime-L}} \cdot t_0 
- {2\sqrt{c} \over L^\prime-L} t_0^2 = \delta \; . \]
The final inequality above follows from $\gamma_1^\prime (0) = 
\langle \vec{T}, \vec{e}_1 \rangle > \sqrt{{8\sqrt{c} \delta \over 
L^\prime -L }}$, and the final equality follows from the definition of 
$t_0$.  This is a contradiction, since the vertical projection to the 
$x_1x_2$-plane of the geodesic $\gamma(t) \subset \hat{M} \subset M$ is contained 
in $D$.  This proves Claim 6.  
\end{proof}

Note that $M_{comp}$ is one component of $M \cap \{(x_1,x_2,x_3) \in \bfR^3 \, | \, 
x_3 \in [L,L^\prime] \}$ and thus $M_{comp} = M_{comp}(L^\prime)$ depends 
on $L^\prime$.  We now wish to increase 
$M_{comp}$ to a connected noncompact surface $\tilde{M}$ that is independent 
of $L^\prime$.  Define 
\[ \tilde{M} := \cup_{L^\prime > L} \; M_{comp}(L^\prime) \; . \] 
Thus $M_{comp}(L^\prime) \subset \tilde{M}$ for all 
$L^\prime$, and $\tilde{M}$ is a disk bounded by one curve in 
$P_L$ and by two upward-pointing vertical rays $r_1,r_2$ with endpoints in $P_L$.  
Since $L^\prime > L + 
\mbox{max}(1,8 \sqrt{c} \delta)$ was arbitrary in the proof of Claim 6, 
an easy corollary of Claim 6 is the following:  

\begin{quote}
{\bf Claim 7:} 
The normal vector $\vec{N}$ on $\tilde{M}$ converges to $\pm \vec{e}_1$ at the 
end of $\tilde{M}$.  More precisely, for 
all $\rho \in (0,1)$, there exists ${\cal L}(\rho)>0$ such that at all 
points $q \in \{(x_1,x_2,x_3) \in \tilde{M} \, | \, x_3>{\cal L}(\rho) \}$, 
the normal $\vec{N}_q$ satisfies 
$||\vec{N}_q - \vec{e}_1|| < \rho$ or $||\vec{N}_q + \vec{e}_1|| < \rho$.  
\end{quote}

\begin{proof}
We choose $s$ so that $L^\prime = 2s$.  By Claim 6, if 
$L^\prime > L + \max (1,8 \sqrt{c} \delta )$, then 
\[ \langle \vec{N}_q,\vec{e}_1 \rangle^2 \geq 1 - {8 \sqrt{c} \delta 
\over 2s-L} \] for every point $q \in P_s \cap M_{comp}$.  Define 
\[ \vec{N}_q^\perp := \vec{N}_q - \langle \vec{N}_q,\vec{e}_1 \rangle 
\vec{e}_1 \; , \] Then $|| \vec{N}_q^\perp ||^2 \leq {8 \sqrt{c} \delta \over 
2s-L}$ and $\vec{N}_q \pm \vec{e}_1 = (\langle \vec{N}_q,\vec{e}_1 
\rangle \pm 1)\vec{e}_1 + \vec{N}_q^\perp$.  
By a straightforward computation, choosing 
\[ s > {16 \sqrt{c} \delta + 3 L \over 3 \rho^2} +1 \] is sufficient to ensure 
\[ \mbox{min}||\vec{N}_q \pm \vec{e}_1|| < \rho \; \mbox{ and } 
L^\prime > L + \max (1,8 \sqrt{c} \delta ) \; . \] Claim 7 is shown.  
\end{proof}

Using Claim 7 and elementary properties of conjugation, we now prove Lemma 4.1.  

Note that dist$_{\bfR^3}$($r_1,r_2$)$= {k \over \ell}$ for some positive integer $k$.  
By Claim 7 and the original construction of the boundary data (i.e. the choices 
we made for the $A_i,B_j$) in the Jenkins-Serrin graph, we see that $k=1$.  
Furthermore, by Claim 7, we have 
\begin{equation}\label{finaleqn}
 \mbox{dist}_{\tilde{M}}(r_1,r_2) = \mbox{dist}_{\bfR^3}(r_1,r_2) = 
{1 \over \ell}  \; . \end{equation} 
Let $\tilde{M}_{conj}$ be the conjugate surface of $\tilde{M}$.  
We have the following properties: 
\begin{enumerate}
\item Since conjugation is an isometry, $\tilde{M}_{conj}$ is 
bounded by one smooth curve of finite length, and two smooth curves 
$\hat{r}_1,\hat{r}_2$ of infinite length.  
\item Since 
conjugation maps straight lines to planar geodesics, 
$\hat{r}_1,\hat{r}_2$ are two boundary planar geodesics of 
$\tilde{M}_{conj}$ that are the images of the boundary rays $r_1,r_2$, respectively, 
under conjugation.  
\item Since conjugation preserves the Gauss map and hence also $\vec{N}$, 
$\hat{r}_1$ and $\hat{r}_2$ each lie in a horizontal plane.  We call these 
two horizontal planes $\hat{P}_1$ and $\hat{P}_2$, respectively.  
\item Since the normal vector $\vec{N}$ is preserved under conjugation, 
$\vec{N}$ on $\tilde{M}_{conj}$ converges to 
$\pm \vec{e}_1$ at the end of $\tilde{M}_{conj}$.  
\item By property 4 above, $\mbox{dist}_{\bfR^3}(\hat{P}_1,\hat{P}_2) = 
\mbox{dist}_{\tilde{M}_{conj}}(\hat{r}_1,\hat{r}_2)$.  
\item Since conjugation is an isometry, 
$\mbox{dist}_{\tilde{M}_{conj}}(\hat{r}_1,\hat{r}_2) = 
\mbox{dist}_{\tilde{M}}(r_1,r_2)$.  
\end{enumerate}
Finally, from equation (\ref{finaleqn}) and properties 5 and 6 above, we conclude: 

\begin{quote}{\bf Claim 8:}
$\mbox{dist}_{\bfR^3}(\hat{P}_1,\hat{P}_2) = {1 \over \ell}$.  
\end{quote}

On the conjugate $\tilde{M}_{conj}$ of $\tilde{M} \subset 
\{(x_1,x_2,x_3) \in M \, | \, x_3 \geq L \}$, the period problem at the end 
is a vertical translation comprised of one reflection through $\hat{P}_1$ composed 
with one reflection through $\hat{P}_2$.  Thus the period problem is a vertical 
translation of length exactly $2 \over \ell$, by Claim 8.  Likewise, the same holds 
for the conjugate surface of any other components of 
$\{(x_1,x_2,x_3) \in M \, | \, x_3 \geq L \}$ and any components of 
$\{(x_1,x_2,x_3) \in M \, | \, x_3 \leq -L \}$ as well, when $L$ is chosen large enough.  
Since the boundary behavior alternates between $+\infty$ and $-\infty$ along 
the alternating $A_i$'s and $B_j$'s, the normal 
vector of the graph $u$ must alternately 
approach $+\vec{e}_1$ and $-\vec{e}_1$ along the $A_i$'s and $B_j$'s.  
Therefore, as one travels along the line segment $D \cap \{(0,x_2) \in \bfR^2\}$,  
the vertical direction of the translation periods at the 
ends of the conjugate surface of $M$ alternates between upward and downward 
translations of length $2 \over \ell$.  

Thus Lemma 4.1 is shown.  


\bibliographystyle{plain}

\newpage

\begin{figure}
\begin{center}
  \includegraphics[width=12.5cm]{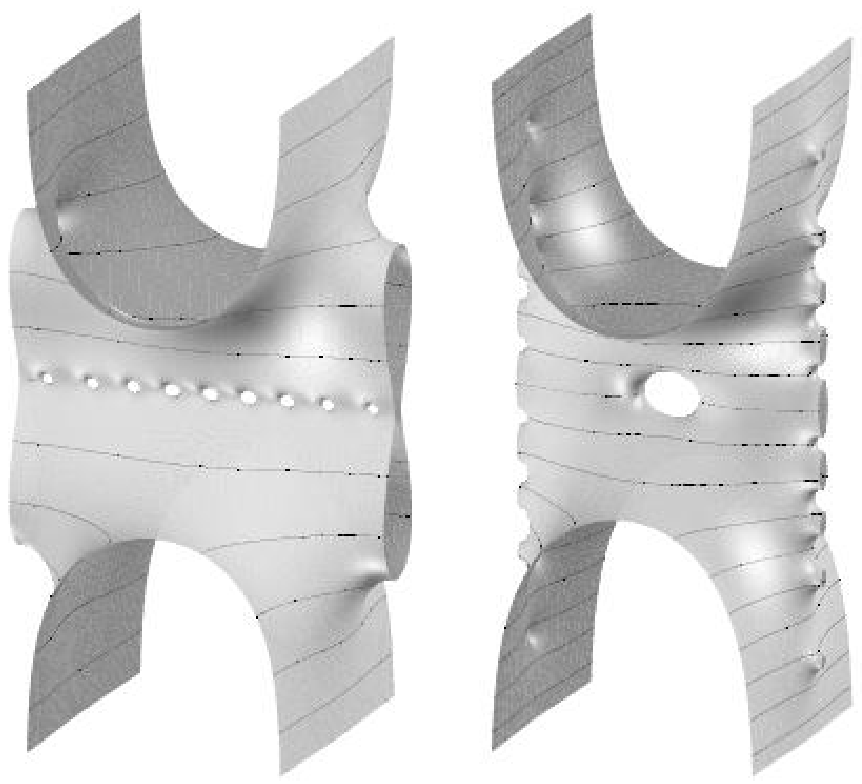} 
\end{center}
  \caption{Fundamental pieces of $M_1^{9+}$ (left) 
           and $M_1^{+,8-}$ (right).}           \label{fig:2k+_imageimage}
\end{figure}

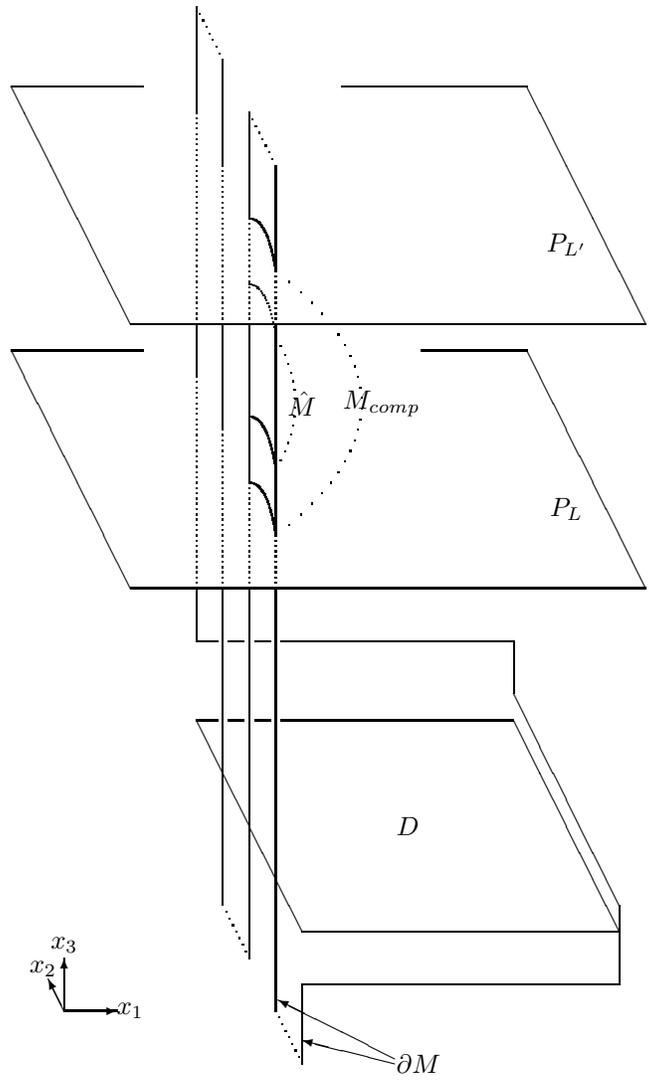
\begin{figure}
\begin{center}
\unitlength=1.0pt
\begin{picture}(-200.00,200.00)(0.00,300.00)
\put(-10.00,410.00){\makebox(0,0)[cc]{$P_{L^\prime}$}}
\put(-10.00,310.00){\makebox(0,0)[cc]{$P_{L}$}}
\put(-110.00,350.00){\makebox(0,0)[cc]{$\hat{M}$}}
\put(-80.00,350.00){\makebox(0,0)[cc]{$M_{comp}$}}
\put(-70.00,190.00){\makebox(0,0)[cc]{$D$}}
\put(-66.00,100.00){\makebox(0,0)[cc]{$\partial M$}}
\put(-175.00,120.00){\makebox(0,0)[cc]{$x_1$}}
\put(-208.00,136.00){\makebox(0,0)[cc]{$x_2$}}
\put(-200.00,145.00){\makebox(0,0)[cc]{$x_3$}}
\put(-200,120){\vector(1,0){20}}
\put(-200,120){\vector(0,1){20}}
\put(-200,120){\vector(-1,2){6}}
\put(-75,100){\vector(-4,1){35}}
\put(-75,102){\vector(-2,1){44}}
\put(-175,380){\line(1,0){195}}
\put(-175,380){\line(-1,2){45}}
\put(-220,470){\line(1,0){50}}
\put(-25,470){\line(-1,0){70}}
\put(20,380){\line(-1,2){45}}
\put(-175,280){\line(1,0){195}}
\put(-175,280){\line(-1,2){45}}
\put(-220,370){\line(1,0){50}}
\put(-25,370){\line(-1,0){40}}
\put(20,280){\line(-1,2){45}}
\put(-110,150){\line(1,0){120}}
\put(-110,150){\line(-1,2){40}}
\put(-150,230){\line(1,0){8}}
\put(-138,230){\line(1,0){6}}
\put(-128,230){\line(1,0){6}}
\put(-30,230){\line(-1,0){88}}
\put(-150,260){\line(1,0){8}}
\put(-138,260){\line(1,0){6}}
\put(-128,260){\line(1,0){6}}
\put(-30,260){\line(-1,0){88}}
\put(10,150){\line(-1,2){40}}
\put(-110,100){\line(0,1){30}}
\put(-110,130){\line(1,0){120}}
\put(10,130){\line(0,1){30}}
\put(10,160){\line(-1,2){40}}
\put(-30,240){\line(0,1){20}}
\put(-150,260){\line(0,1){20}}
\put(-140,160){\line(0,1){120}}
\put(-130,140){\line(0,1){140}}
\put(-120,120){\line(0,1){160}}
\put(-120,300){\line(0,1){80}}
\put(-130,320){\line(0,1){60}}
\put(-140,340){\line(0,1){40}}
\put(-150,360){\line(0,1){20}}
\put(-120,400){\line(0,1){40}}
\put(-130,420){\line(0,1){40}}
\put(-140,440){\line(0,1){40}}
\put(-150,460){\line(0,1){40}}
\bezier8(-120,440)(-125,450)(-130,460)
\bezier8(-140,480)(-145,490)(-150,500)
\bezier8(-110,100)(-115,110)(-120,120)
\bezier8(-130,140)(-135,150)(-140,160)
\bezier10(-120,380)(-120,390)(-120,400)
\bezier20(-130,380)(-130,400)(-130,420)
\bezier30(-140,380)(-140,410)(-140,440)
\bezier40(-150,380)(-150,420)(-150,460)
\bezier10(-120,280)(-120,290)(-120,300)
\bezier20(-130,280)(-130,300)(-130,320)
\bezier30(-140,280)(-140,310)(-140,340)
\bezier40(-150,280)(-150,320)(-150,360)
\bezier200(-120,300)(-123,320)(-130,320)
\bezier200(-120,325)(-123,345)(-130,345)
\bezier20(-120,375)(-123,395)(-130,395)
\bezier200(-120,400)(-123,420)(-130,420)
\bezier15(-120,325)(-105,350)(-120,375)
\bezier25(-120,300)(-55,350)(-120,400)
\end{picture}
\vspace{4in}
\end{center}
    \caption{The location of $\hat{M}$, as defined in Claim 5.}%
    \label{fig:lasta}
\end{figure}            

\newpage

\begin{figure}
\begin{center}
\unitlength=1.0pt
\begin{picture}(-200.00,200.00)(0.00,300.00)
\put(-45.00,295.00){\makebox(0,0)[cc]{$P_{L^\prime}$}}
\put(0.00,120.00){\makebox(0,0)[cc]{$P_L$}}
\put(-108.00,135.00){\makebox(0,0)[cc]{$M_{comp}$}}
\put(-17.00,162.00){\makebox(0,0)[cc]{$M_{comp}^{ext}$}}
\put(-150,100){\line(1,0){180}}
\put(-103,140){\vector(1,1){18}}
\put(-32,162){\vector(-1,0){25}}
\put(-210,220){\line(1,0){50}}
\bezier8(-160,220)(-140,220)(-120,220)
\bezier7(-64,220)(-47,220)(-30,220)
\put(-150,100){\line(-1,2){60}}
\put(30,100){\line(-1,2){50}}
\bezier5(-20,200)(-25,210)(-30,220)
\put(-150,200){\line(1,0){180}}
\put(-210,320){\line(1,0){180}}
\put(-150,200){\line(-1,2){60}}
\put(30,200){\line(-1,2){60}}
\put(-61,102){\line(0,1){100}}
\put(-68,116){\line(0,1){84}}
\bezier8(-68,200)(-68,208)(-68,216)
\put(-75,130){\line(0,1){70}}
\bezier15(-75,200)(-75,215)(-75,230)
\put(-82,144){\line(0,1){56}}
\bezier22(-82,200)(-82,222)(-82,244)
\put(-89,158){\line(0,1){42}}
\bezier29(-89,200)(-89,229)(-89,258)
\put(-96,172){\line(0,1){28}}
\bezier36(-96,200)(-96,236)(-96,272)
\put(-103,186){\line(0,1){14}}
\bezier43(-103,200)(-103,243)(-103,286)
\bezier50(-110,200)(-110,250)(-110,300)
\bezier50(-117,214)(-117,264)(-117,314)
\bezier200(-61,102)(-68,106)(-68,116)
\bezier200(-68,116)(-68,126)(-75,130)
\bezier200(-75,130)(-82,134)(-82,144)
\bezier200(-82,144)(-82,154)(-89,158)
\bezier20(-82,159)(-82,169)(-89,173)
\bezier20(-82,174)(-82,184)(-89,188)
\bezier20(-82,189)(-82,199)(-89,203)
\bezier8(-82,204)(-82,214)(-89,218)
\bezier8(-82,224)(-82,234)(-89,238)
\bezier200(-89,158)(-96,162)(-96,172)
\bezier200(-96,172)(-96,182)(-103,186)
\bezier200(-103,186)(-110,190)(-110,200)
\bezier10(-110,200)(-110,210)(-117,214)
\bezier200(-61,202)(-68,206)(-68,216)
\bezier200(-68,216)(-68,226)(-75,230)
\bezier200(-75,230)(-82,234)(-82,244)
\bezier200(-82,244)(-82,254)(-89,258)
\bezier200(-89,258)(-96,262)(-96,272)
\bezier200(-96,272)(-96,282)(-103,286)
\bezier200(-103,286)(-110,290)(-110,300)
\bezier200(-110,300)(-110,310)(-117,314)
\end{picture}
\vspace{4in}
\end{center}
    \caption{The construction of $M_{comp}^{ext}$.}%
    \label{fig:lastb}
\end{figure}
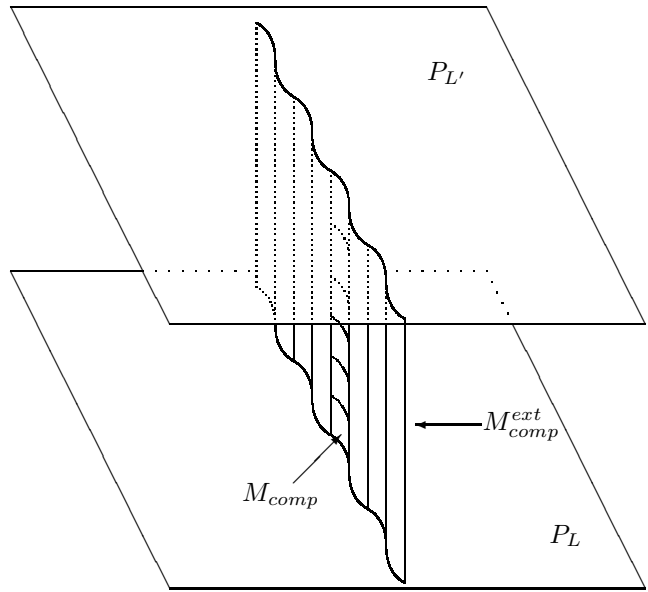            

\newpage

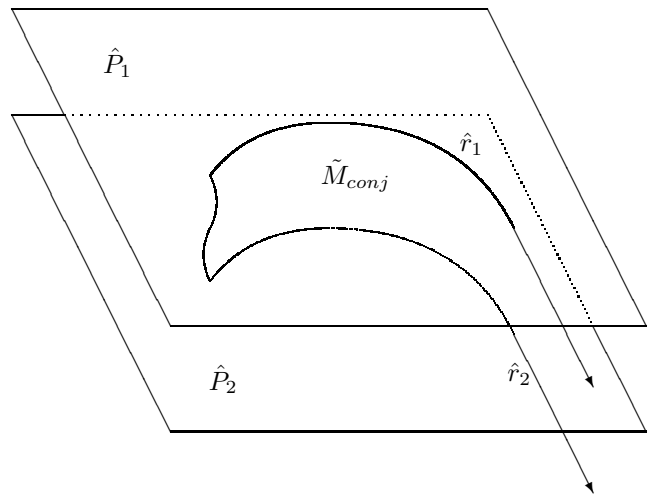
\begin{figure}
\begin{center}
\unitlength=1.0pt
\begin{picture}(-200.00,200.00)(0.00,300.00)
\put(-170.00,240.00){\makebox(0,0)[cc]{$\hat{P}_1$}}
\put(-130.00,120.00){\makebox(0,0)[cc]{$\hat{P}_2$}}
\put(-80.00,197.00){\makebox(0,0)[cc]{$\tilde{M}_{conj}$}}
\put(-36.00,209.00){\makebox(0,0)[cc]{$\hat{r}_1$}}
\put(-18.00,122.00){\makebox(0,0)[cc]{$\hat{r}_2$}}
\put(-150,140){\line(1,0){180}}
\put(-210,260){\line(1,0){180}}
\put(-150,140){\line(-1,2){60}}
\put(30,140){\line(-1,2){60}}
\put(-150,100){\line(1,0){180}}
\put(-210,220){\line(1,0){20}}
\bezier50(-190,220)(-110,220)(-30,220)
\put(-150,100){\line(-1,2){60}}
\put(30,100){\line(-1,2){20}}
\bezier50(10,140)(-10,180)(-30,220)
\bezier100(-135,157)(-120,177)(-90,177)
\bezier100(-20,137)(-40,177)(-90,177)
\put(-20,137){\vector(1,-2){30}}
\bezier1000(-135,197)(-120,217)(-90,217)
\bezier1000(-20,177)(-40,217)(-90,217)
\put(-20,177){\vector(1,-2){30}}
\bezier50(-135,197)(-130,187)(-135,177)
\bezier50(-135,157)(-140,167)(-135,177)
\end{picture}
\vspace{4in}
\end{center}
    \caption{The conjugate surface of $\tilde{M}$.}%
    \label{fig:lastc}
\end{figure}            

\newpage

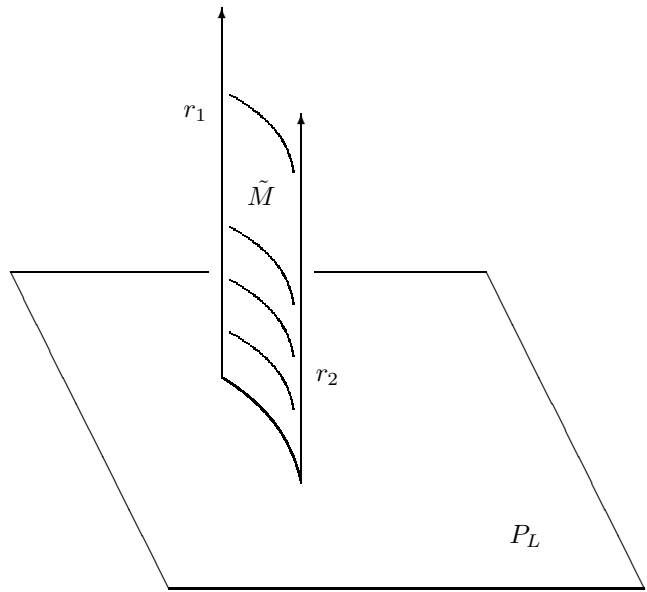
\begin{figure}
\begin{center}
\unitlength=1.0pt
\begin{picture}(-200.00,200.00)(0.00,300.00)
\put(-150.00,280.00){\makebox(0,0)[cc]{$r_1$}}
\put(-125.00,250.00){\makebox(0,0)[cc]{$\tilde{M}$}}
\put(-100.00,180.00){\makebox(0,0)[cc]{$r_2$}}
\put(-25.00,120.00){\makebox(0,0)[cc]{$P_L$}}
\put(-140,180){\vector(0,1){140}}
\put(-110,140){\vector(0,1){140}}
\put(-160,100){\line(1,0){180}}
\put(-220,220){\line(1,0){75}}
\put(-105,220){\line(1,0){65}}
\put(-160,100){\line(-1,2){60}}
\put(20,100){\line(-1,2){60}}
\bezier1000(-140,180)(-115,165)(-110,140)
\bezier50(-137,197)(-115,185)(-113,168)
\bezier50(-137,217)(-115,205)(-113,188)
\bezier50(-137,237)(-115,225)(-113,208)
\bezier50(-137,287)(-115,275)(-113,258)
\end{picture}
\vspace{4in}
\end{center}
    \caption{The surface $\tilde{M}$.}%
    \label{fig:lastd}
\end{figure}            

\end{document}